\documentclass[11pt]{amsart}
\usepackage{amssymb}
\usepackage{amsthm}
\usepackage{amsmath}
\usepackage{latexsym}
\usepackage{comment}
\usepackage{hyperref}
\usepackage{cleveref}
\usepackage{verbatim}
\usepackage{xypic}
\usepackage{graphicx}
\usepackage{subcaption}
\usepackage[utf8]{inputenc}
\usepackage[margin=1in]{geometry}
\usepackage[shortlabels]{enumitem}
\usepackage{xcolor}

\newtheorem{thm}{Theorem}[section]
\newtheorem{prop}[thm]{Proposition}

\newtheorem{lemma}[thm]{Lemma}
\newtheorem{cor}[thm]{Corollary}
\newtheorem{dfn}[thm]{Definition}

\theoremstyle{definition} \newtheorem{ex}[thm]{Example}
\newtheorem{question}[thm]{Question}

\theoremstyle{definition} \newtheorem{rmk}[thm]{Remark}
\newtheorem{algo}[thm]{Algorithm}

\newcommand{\cc}{\mathbb{C}}
\newcommand{\rr}{\mathbb{R}}
\newcommand{\qq}{\mathbb{Q}}
\newcommand{\zz}{\mathbb{Z}}

\newcommand{\proj}{\mathbb{P}}

\newcommand{\GL}{\mathrm{GL}}

\newcommand{\PGL}{\mathrm{PGL}}

\newcommand{\Berk}{\mathbb{P}_{\cc_K}^{1, \mathrm{an}}}
\newcommand{\Smin}{S^{\mathrm{min}}}

\graphicspath{{pics/}}

\title[Branch points of split degenerate superelliptic curves I]{Branch points of split degenerate superelliptic curves I: construction of Schottky groups}
\author{Jeffrey Yelton}
\address{\parbox{\linewidth}{Department of Mathematics and Computer Science, Wesleyan University \\ 265 Church Street, Middletown, CT 06459-0128}}
\email{jyelton@wesleyan.edu}

\begin{document}

\maketitle

\begin{abstract}

Let $K$ be a field with a discrete valuation, and let $p$ be a prime.  It is known that if $\Gamma \lhd \Gamma_0 < \PGL_2(K)$ is a Schottky group normally contained in a larger group which is generated by order-$p$ elements each fixing $2$ points $a_i, b_i \in \proj_K^1$, then the quotient of a certain subset of the projective line $\proj_K^1$ by the action of $\Gamma$ can be algebraized as a superelliptic curve $C : y^p = f(x) / K$.  The subset $S \subset K \cup \{\infty\}$ consisting of these pairs $a_i, b_i$ of fixed points is mapped modulo $\Gamma$ to the set of branch points of the superelliptic map $x : C \to \proj_K^1$.  We produce an algorithm for determining whether an input even-cardinality subset $S \subset K \cup \{\infty\}$ consists of fixed points of generators of such a group $\Gamma_0$ and which, in the case of a positive answer, modifies $S$ into a subset $\Smin \subset K \cup \{\infty\}$ with particularly nice properties.  Our results do not involve any restrictions on the prime $p$ or on the residue characteristic of $K$ and allow these to be the same.

\end{abstract}

\section{Introduction} \label{sec 1}

This paper concerns the non-archimedean uniformization of \textit{superelliptic curves} which are $p$-cyclic covers of the projective line for some prime $p$ different from the characteristic of the ground field; such an object is a smooth curve $C$ over a field $K$ of characteristic different from $p$, equipped with a surjective degree-$p$ morphism $C \to \proj_K^1$.  We denote the set of branch points of this surjective morphism by $\mathcal{B}$ and write $d = \#\mathcal{B}$.  One can easily compute from the Riemann-Hurwitz formula that the genus of $C$ is given by $\frac{1}{2}(p - 1)(d - 2)$.  It is also well known that a superelliptic curve which is a degree-$p$ cover of the projective line can be described by an affine model of the form 
\begin{equation} \label{eq superelliptic model}
y^p = \prod_{i = 1}^{d'} (x - z_i)^{r_i} \in K[x],
\end{equation}
 with $d' \in \{d, d - 1\}$, where $1 \leq r_i \leq p - 1$ for $1 \leq i \leq d'$ and $\mathcal{B} = \{z_1, \dots, z_{d'}\}$ (resp. $\mathcal{B} = \{z_1, \dots, z_{d'}, \infty\}$) if $d' = d$ (resp. $d' = d - 1$).  In the special case that $p = 2$, we call $C$ a \textit{hyperelliptic curve}; in this case, our formula for the genus $g$ implies that $d$ is even and that the equation in (\ref{eq superelliptic model}) can be written as $y^2 = f(x) \in K[x]$ for a squarefree polynomial $f$ of degree $2g + 1$ (resp. $2g + 2$) if the degree-$2$ cover $C \to \proj_K^1$ is not (resp. is) branched above $\infty$.
 
\subsection{Superelliptic curves over discrete valuation fields and Schottky groups} \label{sec 1 superelliptic Schottky}

Throughout this paper, we assume that $K$ is a field equipped with a discrete valuation $v : K^\times \to \zz$, and we denote by $\cc_K$ the completion of an algebraic closure of $K$.  We fix a prime $p$ and a primitive $p$th root of unity $\zeta_p \in \cc_K$ and assume throughout that we have $\zeta_p \in K$.  (This will ensure, among other things, that each automorphism of $C / K$ as a cyclic $p$-cover of $\proj_K^1$ is defined over $K$.)  We adopt the convention of using the notation $\proj_K^1$ both for the projective line with its structure as a variety over $K$ and for the set of $K$-points of $\proj_K^1$, \textit{i.e.} in a context that will appear frequently in this paper, we write $\proj_K^1$ for the set $K \cup \{\infty\}$.  We follow a similar convention for the notation $\proj_{\cc_K}^1$.

When $C$ is an elliptic curve, which is the case where $p = 2$ and $g = 1$, it was shown by Tate in \cite{tate1995review} that $C$ can be uniformized as the quotient of $K^{\times}$ by the subgroup generated by a single element $q$ of positive valuation if and only if $C$ has split multiplicative reduction over $K$; more precisely, for some element $q \in K^\times$ satisfying $v(q) = -j(C) > 0$, there is a rigid analytic isomorphism of groups $C(K) \stackrel{\sim}{\to} K^{\times} / \langle q \rangle$.

A generalization of this result has been found for any curve $C$ (not necessarily superelliptic) of any genus $g \geq 1$ by Mumford in \cite{mumford1972analytic}: namely, a curve $C$ of genus $g$ can be realized as a quotient of a certain subset of $\proj_K^1$ by the action of a free subgroup $\Gamma < \PGL_2(K)$ of $g$ generators if and only if the curve $C / K$ satisfies a property called \textit{split degenerate reduction} (see \cite[Definition 6.7]{papikian2013non} or \cite[\S IV.3]{gerritzen2006schottky}).  The free subgroup $\Gamma < \PGL_2(K)$ is called a \textit{Schottky group} (see \Cref{dfn schottky} below); letting $\Omega \subset \proj_K^1$ be the complement in $\proj_K^1$ of the limit points of $\proj_K^1$ under the action of $\Gamma$, the Schottky group $\Gamma$ acts discontinuously on $\Omega$ in the usual way by fractional linear transformations, and the curve $C$ is uniformized as $\Omega / \Gamma$.  This main result on non-archimedean uniformization of curves is given as \cite[Theorem 4.20]{mumford1972analytic} and \cite[Theorems III.2.2, III.2.12.2, and IV.3.10]{gerritzen2006schottky}; many more details are contained in those sources.  We comment that in the special case of $g = 1$, after applying an appropriate automorphism of $\proj_K^1$ we get $\Omega = \proj_K^1 \smallsetminus \{0, \infty\} = K^\times$ and that $\Gamma$ is generated by the fractional linear transformation $z \mapsto qz$ for some element $q \in K^\times$ of positive valuation, and thus we recover the Tate uniformization established in \cite{tate1995review}.

It is shown in \cite[\S9.2]{gerritzen2006schottky} and \cite[\S1]{van1983non} (for the $p = 2$ case) and in \cite[\S2]{van1982galois} (for general $p$) that given a prime $p$ and a split degenerate curve $C / K$ of genus $(p - 1)g$ realized as such a quotient $\Omega / \Gamma$, the curve $C$ is superelliptic and a degree-$p$ cover of $\proj_K^1$ if and only if $\Gamma$ is normally contained in a larger subgroup $\Gamma_0 < \PGL_2(K)$ generated by $g + 1$ elements $s_0, \dots, s_g$ whose only relations are $s_0^p = \dots = s_g^p = 1$.  In this situation, we have $[\Gamma_0 : \Gamma] = p$ and $\Omega / \Gamma_0 \cong \proj_K^1$ so that the natural surjection $\Omega / \Gamma \twoheadrightarrow \Omega / \Gamma_0$ is just the degree-$p$ covering map $C \to \proj_K^1$.  Each order-$p$ element $s_i \in \PGL_2(K)$ fixes exactly $2$ points of $\proj_K^1$, which we denote as $a_i$ and $b_i$.  Writing $S = \{a_0, b_0, \dots, a_g, b_g\} \subset \proj_K^1$, it is easy to verify that the set-theoretic image of $S$ modulo the action of $\Gamma_0$ coincides with the set of branch points $\mathcal{B} \subset \proj_K^1 \cong \Omega / \Gamma_0$.

Conversely, one may start with a subset $S \subset \proj_K^1$ of cardinality $2g + 2$ and try to construct a split degenerate superelliptic curve over $K$ in the following manner.  Suppose that we have labeled the elements of $S$ as $a_0, b_0, \dots, a_g, b_g$.  For $0 \leq i \leq g$, there is an element $s_i \in \PGL_2(K)$ of order $p$, unique up to power (see \Cref{prop order-p elliptic}(a) below), which fixes the points $a_i, b_i \in \proj_K^1$.  We define subgroups $\Gamma < \Gamma_0 < \PGL_2(K)$ as 
\begin{equation} \label{eq Gamma}
\Gamma_0 = \langle s_0, \dots, s_g \rangle, \ \Gamma = \langle \gamma_{i, j} := s_0^{j - 1} s_i s_0^{-j} \rangle_{0 \leq i \leq g, \ 1 \leq j \leq p - 1}.
\end{equation}
It is an elementary group-theoretic exercise to show from these definitions that we have $\Gamma \lhd \Gamma_0$.

If the subgroup $\Gamma$ that we have constructed in this way is Schottky, then letting $\Omega \subset \proj_K^1$ be the complement of the limit points of $\Gamma$ (as in \Cref{dfn limit points} below), the group $\Gamma_0$ is the desired subgroup of $\PGL_2(K)$ such that $[\Gamma_0 : \Gamma] = p$ and $\Omega / \Gamma_0 \cong \proj_K^1$, so that the quotient $\Omega / \Gamma$ can be realized as a split degenerate superelliptic curve $C$ over $K$.

In general, it is not always possible to construct a Schottky group and a superelliptic curve from such a subset $S \subset \proj_K^1$ in this way, as the group $\Gamma < \PGL_2(K)$ defined as in (\ref{eq Gamma}) for order-$p$ elements $s_i \in \PGL_2(K)$ fixing points $a_i, b_i \in S$ as above may turn out not to be Schottky for any labeling of the elements of $S$ as $a_0, b_0, \dots, a_g, b_g$.  As we will show in \Cref{rmk clustered in pairs}(a), a necessary condition on $S$ is that it be \textit{clustered in pairs}\footnote{In fact, when $p$ equals the residue characteristic of $K$, \Cref{rmk clustered in pairs}(a) affirms an additional necessary condition.} as in \Cref{dfn clustered in pairs} below.  In this case there is only one way to label the elements of $S$ as $a_0, b_0, \dots, a_g, b_g$ such the subgroup $\Gamma < \PGL_2(K)$ defined in (\ref{eq Gamma}) may possibly be Schottky.  However, being clustered in pairs is still not a sufficient condition for $S$ to determine a Schottky group in this way: see \Cref{ex clustered in pairs not sufficient} below.

This paper focuses on non-archimedean uniformization of split degenerate superelliptic curves from the point of view of beginning with a subset $S \subset \proj_K^1$ and investigating the subgroups $\Gamma \lhd \Gamma_0 < \PGL_2(K)$ which can be computed from $S$ in the manner described above.  As a convenient reference, below we make precise the relationship between such subsets and such groups, as well as list some related notation that will be used throughout the paper.

\begin{dfn} \label{dfn associated to S}

Let $S \subset \proj_K^1$ be a $(2g + 2)$-element subset for some $g \geq 1$.  Suppose that we have fixed a labeling of the elements of $S$ as $a_0, b_0, \dots, a_g, b_g \in \proj_K^1$ (noting that if $S$ is clustered in pairs as in \Cref{dfn clustered in pairs}, then \Cref{rmk clustered in pairs}(a) says that there is a canonical partitioning of $S$ into cardinality-$2$ subsets consisting of points labeled $a_i, b_i$).  Then we use the notation specified below for the following objects determined by $S$ (and this labeling of its elements):

\begin{itemize}

\item for $0 \leq i \leq g$, a choice of order-$p$ element $s_i \in \PGL_2(K)$ which fixes the points $a_i, b_i \in \proj_K^1$;

\item the subgroups $\Gamma \lhd \Gamma_0 < \PGL_2(K)$ given by (\ref{eq Gamma}) in terms of the above elements $s_i \in \PGL_2(K)$;

\item the axis of the Berkovich projective line (see \S\ref{sec 2 berk} below) which connects the points $\eta_{a_i}, \eta_{b_i} \in \Berk$ of Type I associated to $a_i, b_i \in \proj_K^1$, which we denote by $\Lambda_{(i)} \subset \Berk$;

\item the tubular neighborhood $B(\Lambda_{(i)}, \frac{v(p)}{p - 1}) \subset \Berk$ of the axis $\Lambda_{(i)}$ (see \S\ref{sec 2 berk} below), which we denote by $\hat{\Lambda}_{(i)} \subset \Berk$.

\end{itemize}

We say that the above objects are \emph{associated to} $S$.

\end{dfn}

We are interested in identifying when a particular subgroup $S \subset \proj_K^1$ yields a superelliptic curve, which inspires the following definition.

\begin{dfn} \label{dfn superelliptic set}

We say that a $(2g + 2)$-element subset $S \subset \proj_K^1$ is ($p$-)\emph{superelliptic} if for some labeling of the elements of $S$ as $a_0, b_0, \dots, a_g, b_g$, the following properties hold for the subgroups $\Gamma \lhd \Gamma_0 < \PGL_2(K)$ associated to $S$ (with respect to our fixed choice of prime $p$):

\begin{enumerate}[(i)]
\item the group $\Gamma$ is Schottky; and 
\item the group $\Gamma_0$ cannot be generated by fewer than $g + 1$ elements.
\end{enumerate}

\end{dfn}

A \emph{superelliptic} subset $S \subset \proj_K^1$ is therefore one whose associated subgroup $\Gamma < \PGL_2(K)$ induces a superelliptic curve $C \cong \Omega / \Gamma$ and such that there is no ``redundancy'' among the elements of $S$; this second condition implies that the genus of $C$ equals $(p - 1)g$ via \Cref{rmk free} below.  (We caution that except in the hyperelliptic curve case where $p = 2$, the letter $g$ used to denote $\frac{1}{2}\#S - 1$ therefore does \textit{not} equal the genus of the resulting curve.)

\subsection{Our main results} \label{sec 1 results}
 
Two questions naturally come up concerning the construction of Schottky groups and genus-$g$ split degenerate hyperelliptic curves from a cardinality-$(2g + 2)$ subset $S \subset \proj_K^1$.

\begin{question} \label{question criterion for being superelliptic}

How may we directly determine whether a subset $S \subset \proj_K^1$ is superelliptic?

\end{question}

\begin{question} \label{question GvanderP}

Suppose that we are given a superelliptic subset $S \subset \proj_K^1$.  Define the subset $\Omega \subset \proj_K^1$, the split degenerate hyperelliptic curve $C \cong \Omega / \Gamma$, and the set of branch points $\mathcal{B}$ as above.  Is it possible to determine the distances (under the metric induced by the valuation $v : K^\times \to \zz$) between the elements of $\mathcal{B}$ only from knowing the distances between the elements of $S$, and if so, how are these distances related?

\end{question}

\Cref{question criterion for being superelliptic} is brought up in the case of hyperelliptic curves both on \cite[p. 279]{gerritzen2006schottky} and in Samuel Kadziela's dissertation \cite{kadziela2007rigid}: it is explained in \cite[\S5.3]{kadziela2007rigid} that the methods used in that dissertation (which make heavy use of \textit{good fundamental domains} for Schottky groups) do not in any obvious way lead to an answer to \Cref{question criterion for being superelliptic}.  The main progress on the question since Gerritzen and van der Put's book came out is found in \cite{kadziela2007rigid}.  In this work, Kadziela considered the special case where, after applying a fractional linear transformation to a $(2g + 2)$-element subset $S \subset \proj_K^1$, we may write 
\begin{equation*}
S = \{a_0 := 0, b_0, a_1, b_1, \dots, a_g = 1, b_g := \infty\}
\end{equation*}
 such that we have 
\begin{equation} \label{eq Kadziela's condition}
v(b_0) > v(a_1) \geq v(b_1) \geq v(a_2) \geq \dots \geq v(b_{g - 1}) > 0,
\end{equation}
 and such that the discs $D_{a_i, b_i}, D_{2 - a_i, 2 - b_i} \subset \cc_K$ for $1 \leq i \leq g - 1$ are all mutually disjoint, where $D_{w, w'} := \{z \in K \ | \ v(z - w) \geq v(w' - w)\}$ denotes the smallest disc containing distinct elements $w, w' \in K$.  Using good fundamental domains, Kadziela showed as \cite[Theorem 5.7]{kadziela2007rigid} that such a subset $S$ is always $2$-superelliptic.
 
We resolve this question for a general even-cardinality subset $S \subset \proj_K^1$ in a different way that does not involve good fundamental domains.  Our approach is to define an operation which we call a \emph{folding} (in \Cref{dfn folding}) which turns an even-cardinality subset $S \subset \proj_K^1$ into another such set $S'$ of the same cardinality, to show that $S$ is superelliptic if and only if, after a sequence of well-chosen foldings, one obtains a set $\Smin$ which has a desirable property that we call being \emph{optimal} (see \Cref{dfn optimal}), and to show that only finitely many foldings need to be performed to either find that $S$ is not superelliptic or yield an optimal set $\Smin$.  Our main result, which asserts that all of this can be done, is \Cref{thm criterion for being superelliptic}, and it serves as an answer to \Cref{question criterion for being superelliptic}.  We moreover formulate an algorithm (presented as \Cref{algo folding}) which performs this process: that is, it takes as an input an even-cardinality subset $S \subset \proj_K^1$, determines whether or not $S$ is a superelliptic subset, and, in the case that $S$ is superelliptic, computes an optimal set $\Smin$ obtained by applying finitely many foldings to $S$.

Meanwhile, \Cref{question GvanderP} is inspired by a conjecture posed by Gerritzen and van der Put as \cite[p. 282]{gerritzen2006schottky} for the case of hyperelliptic curves under the assumption that $K$ has residue characteristic different from $2$; it was later recalled by Kadziela as \cite[Conjecture 3.1]{kadziela2007rigid}.  We will deal with this conjecture in a sequel paper, in which we will show that (a much more precise version of) Gerritzen and van der Put's conjecture holds under the assumption that $S$ is optimal.  We will moreover prove such a statement in a much more general setting where $p$ and the residue characteristic of $K$ can be independently chosen to be any pair of primes.  This will provide, among other results, a positive answer to a certain variation of \Cref{question GvanderP}.

\subsection{Outline of the paper} \label{sec 1 outline}

Our strategy in obtaining the results described above which address \Cref{question criterion for being superelliptic} (along with our answer to \Cref{question GvanderP} which is to appear) is to approach the problem through the setting of Berkovich theory, in particular by viewing $\proj_{\cc_K}^1$ as a subset of the Berkovich projective line $\Berk$, and considering the action of certain elements of the projective linear group $\PGL_2(K)$ on it.  With this in mind, we dedicate the following section \S\ref{sec 2} to setting up the necessary framework by first introducing the Berkovich projective line $\Berk$ (in \S\ref{sec 2 berk}), then establishing necessary descriptions of how various elements of $\PGL_2(K)$ act on $\Berk$ (in \S\ref{sec 2 actions}), and then developing the theory of Schottky groups and the properties of Schottky groups which are necessary in order for the resulting curve to be a degree-$p$ cover of the projective line (in \S\ref{sec 2 schottky}).  We approach the discussion in \S\ref{sec 2 schottky} in particular from the point of view of beginning with an even-cardinality subset $S \subset \proj_K^1$ which may potentially be used (if $S$ is superelliptic as in \Cref{dfn superelliptic set}) to construct a Schottky group that yields such a superelliptic curve.  We finish this section with \S\ref{sec 2 clustered in pairs}, which discusses the above-mentioned ``clustered in pairs'' property on even-cardinality subsets $S \subset \proj_K^1$ that is necessary (but generally not sufficient) for $S$ to be superelliptic.

In \S\ref{sec 3}, we proceed to introduce the objects and definitions that will allow us to state and prove our main theorem.  We first (in \S\ref{sec 3 convex hull}) define and discuss the properties of the \emph{reduced convex hull} $\Sigma_{S, 0}$ of an even-cardinality subset $S \subset \proj_K^1$.  Then we define and investigate properties of the above-mentioned \emph{folding} operations which can be performed on sets $S$, using an intuitive description of what foldings do to their reduced convex hulls.  Following this, we open \S\ref{sec 3 optimal} by defining what it means for a set $S$ to be \emph{optimal}, at which point we are ready to state and prove our main result, \Cref{thm criterion for being superelliptic} in this same subsection.  We then (in \S\ref{sec 3 special cases}) proceed to use the theory established in the rest of \S\ref{sec 3} to develop some special cases in which we are more easily able to determine whether a set $S$ is optimal or superelliptic.  Included in this is the aformentioned situation covered by Kadziela: see \Cref{rmk kadziela} below.

The final section \S\ref{sec 4} of this paper provides a practical algorithm for determining whether an even-cardinality subset $S \subset \proj_K^1$ is superelliptic and modifying $S$ through a finite number of foldings into an optimal set $\Smin$ if so.  This algorithm is given in \S\ref{sec 4 algorithm}, in which the procedure involved is explained using the theory developed in the previous section \S\ref{sec 3}: see \Cref{algo folding} and \Cref{thm algorithm} (which asserts that the algorithm works and terminates after a finite number of steps) and its proof there.  Finally, in \S\ref{sec 4 examples} we compute a couple of examples using our algorithm.

\subsection{Acknowledgements}

The author would like to thank Christopher Rasmussen and Joseph Silverman for useful discussions that took place while the author was developing the methods used to address Questions \ref{question criterion for being superelliptic} and \ref{question GvanderP}.

\section{Preliminaries on automorphisms of the Berkovich projective line} \label{sec 2}

Given the completion of an algebraic closure $\cc_K$ of $K$, we write $v : \cc_K \to \rr$ for an extension of the valuation $v : K^\times \to \zz$.  Below when we speak of a \emph{disc} $D \subset \cc_K$, we mean that $D$ is a closed disc with respect to the metric induced by $v : \cc_K \to \rr$; in other words, $D = \{z \in \cc_K \ | \ v(z - c) \geq r\}$ for some center $c \in \cc_K$ and real number $r \in \rr$, which is the \emph{(logarithmic) radius} of $D$.  Given a disc $D \subset \cc_K$, we denote its logarithmic radius by $d(D)$.

\subsection{The Berkovich projective line} \label{sec 2 berk}

The \emph{Berkovich projective line} $\Berk$ over an algebraic closure of a discrete valuation field $K$ is a type of rigid analytification of the projective line $\proj_{\cc_K}^1$ and is typically defined in terms of multiplicative seminorms on $\cc_K[x]$ as in \cite[\S1]{baker2008introduction} and \cite[\S6.1]{benedetto2019dynamics}.  Points of $\Berk$ are identified with multiplicative seminorms which are each classified as Type I, II, III, or IV.  For the purposes of this paper, we may safely ignore points of Type IV and need only adopt a fairly rudimentary construction which does not directly involve seminorms.

\begin{dfn} \label{dfn berk}

Define the \textit{Berkovich projective line}, denoted $\Berk$, to be the topological space with points and topology given as follows.  The points of $\Berk$ are identified with  
\begin{enumerate}[(i)]
\item $\cc_K$-points $z \in \proj_{\cc_K}^1$, which we will call \emph{points of Type I}; and 
\item discs $D \subset \cc_K$; if $d(D) \in \qq$ (resp. $d(D) \notin \qq$), we call this a \emph{point of Type II} (resp. a \emph{point of Type III}).
\end{enumerate}

A point of $\Berk$ which is identified with a point $z \in \proj_{\cc_K}^1$ (resp. a disc $D \subset \cc_K$) is denoted $\eta_z \in \Berk$ (resp. $\eta_D \in \Berk$).

We define an infinite metric on $\Berk$ given by the distance function 
\begin{equation*}
\delta : \Berk \times \Berk \to \rr \cup \{\infty\}
\end{equation*}
defined as follows.  We set $\delta(\eta_z, \eta') = \infty$ for any point $\eta_z$ of Type I and any point $\eta' \neq \eta_z \in \Berk$.  Given a containment $D \subseteq D' \subset \cc_K$ of discs, we set $\delta(\eta_D, \eta_{D'}) = d(D) - d(D') \in \rr$.  More generally, if $D, D' \subset \cc_K$ are discs and $D'' \subset \cc_K$ is the smallest disc containing both $D$ and $D'$, we set 
\begin{equation}
\delta(\eta_D, \eta_{D'}) = \delta(\eta_D, \eta_{D''}) + \delta(\eta_{D'}, \eta_{D''}) = d(D) + d(D') - 2d(D'').
\end{equation}

We endow the subspace of $\Berk$ consisting of points of Type II and III with the topology induced by the metric given by $\delta$, and we extend this to a topology on all of $\Berk$ in the following manner\footnote{This is an extension to all of $\Berk$ of what is known as the \emph{strong topology} on the hyperbolic space $\Berk \smallsetminus \{\eta_z\}_{z \in \cc_K}$.}: for each point $z \in \cc_K \subset \Berk$ of Type I, sets of the form 
\begin{equation*}
U_{D, z, r} := \{\eta_z\} \cup \{\eta \in \Berk \ | \delta(\eta, \eta_{D'}) < r, \ z \in D' \subseteq D\}
\end{equation*}
for $r \in \rr_{> 0}$ and discs $D \subset \cc_K$ containing $z$ are open neighborhoods of $\eta_z \in \Berk$; while sets of the form 
\begin{equation*}
U_{D, \infty, r} := \{\eta_\infty\} \cup \{\eta \in \Berk \ | \ \delta(\eta, \eta_{D'}) < r, \ D' \supseteq D\}
\end{equation*}
for $r \in \rr_{> 0}$ and discs $D \subset \cc_K$ are open neighborhoods of $\eta_\infty$.

\end{dfn}

\begin{rmk} \label{rmk paths in berk}

The space $\Berk$, under the topology given in \Cref{dfn berk}, is path-connected.  Letting $E \subset \cc_K$ be a closed disc or a singleton subset, let us write $\eta_E$ for $\eta_E$ (resp. $\eta_z$) if $E$ is a closed disc (resp. a singleton subset $\{z\}$).  Pick an arbitrary pair of points $\eta, \eta' \in \Berk$, where $\eta = \eta_E$ and $\eta' = \eta_{E'}$ for some $E, E' \subset \cc_K$.  Write $\eta \vee \eta' = \eta_D$, where $D \subset \cc_K$ is the smallest closed disc containing both $E$ and $E'$.  If $E \subset \cc_K$ is a closed disc, define the path $\lambda: [-d(E), -d(D)] \to \Berk$ from $\eta$ to $\eta \vee \eta'$ given by sending $r \in [-d(E), -d(D)]$ to the disc of radius $-r$ which contains $E$.  If $E = \{z\}$ for some $z \in K$, define the path $\lambda: [0, e^{d(D)}] \to \Berk$ from $\eta$ to $\eta \vee \eta'$ given by sending $0$ to $\eta = \eta_z$ and sending $s \in (0, e^{d(D)}]$ to the disc of radius $-\ln(s)$ containing $z$.  If $E = \{\infty\}$, define the path $\lambda: [0, e^{-d(D)}] \to \Berk$ from $\eta$ to $\eta \vee \eta'$ given by sending $0$ to $\eta = \eta_\infty$ and sending $s \in (0, -e^{-d(D)}]$ to the disc of radius $\ln(s)$ containing $D$.  Now define a path $\lambda'$ connecting $\eta$ to $\eta' \vee \eta'$ analogously, and the concatenation of $\lambda$ with the reversal of $\lambda'$ furnishes a path from $\eta$ to $\eta'$.

One can show that the path constructed above is in fact the unique path from $\eta$ to $\eta'$ which is one-to-one (\textit{i.e.} such that there is no backtracking).  This property implies that given any point $\eta \in \Berk$ and subspace $\Lambda \subset \Berk$, there is a unique point on $\Lambda$ which is closest to $\eta$ with respect to the distance function $\delta$.  Similarly, given two subspaces $\Lambda, \Lambda' \subset \Berk$, there are unique points $\eta \in \Lambda$ and $\eta' \in \Lambda'$ which are maximally close to each other.  This observation yields a well-defined notion of distance from a point to a subspace as well as the distance between two subspaces of $\Berk$.

\end{rmk}

The above definition and observations allow us to set the following notation.  Below we denote the image in $\Berk$ of the (unique) shortest path between two points $\eta, \eta' \in \Berk$, as defined in the above remark, by $[\eta, \eta'] \subset \Berk$, and we will often refer to this image itself as ``the (shortest) path'' from $\eta$ to $\eta'$; note that with this notation we have $[\eta, \eta'] = [\eta', \eta]$.  Given points $\eta, \eta', \eta'' \in \Berk$, we may speak of the \emph{concatination} of the path $[\eta, \eta']$ with $[\eta', \eta'']$ as simply their union, which is (the image of) a path from $\eta$ to $\eta''$ that coincides with $[\eta, \eta''] \subset \Berk$ if and only if we have $[\eta, \eta'] \cap [\eta', \eta''] = \{\eta'\}$, or in other words, if ``there is no backtracking''.

We refer to the path between distinct points $\eta_a, \eta_b \in \Berk$ of Type I as an \emph{axis} and will sometimes denote the axis connecting them by $\Lambda_{a, b} := [\eta_a, \eta_b] \subset \Berk$.

In light of \Cref{rmk paths in berk}, given a point $\eta \in \Berk$ and subspaces $\Lambda, \Lambda' \in \Berk$, we write $\delta(\eta, \Lambda)$ and $\delta(\Lambda, \Lambda')$ respectively for the distances between $\eta$ and (the closest point to $\eta$ on) $\Lambda'$ and between (the closest points on) $\Lambda$ and $\Lambda'$.

We will also often need to deal with tubular neighborhoods of subspaces of $\Berk$, and so we introduce the following definition and notation.

\begin{dfn} \label{dfn tubular neighborhood}

Given a subspace $\Lambda \subset \Berk$ and a real number $r > 0$, we define the \emph{(closed) tubular neighborhood} of $\Lambda$ of radius $r$ to be 
\begin{equation} \label{eq tubular neighborhood}
B(\Lambda, r) = \{\eta \in \Berk \ | \ \delta(\eta, \Lambda) \leq r\}.
\end{equation}

\end{dfn}

\subsection{Fractional linear transformations} \label{sec 2 actions}

There is a well-known action of the projective linear group $\PGL_2(K) = \GL_2(K) / K^\times$ on the projective line $\proj_K^1$ given by 
\begin{equation}
\begin{bmatrix} a & b \\ c & d \end{bmatrix}: z \mapsto \frac{az + b}{cz + d},
\end{equation}
where $\big[\begin{smallmatrix} a & b \\ c & d \end{smallmatrix}\big] \in \GL_2(K)$ is the matrix representing some element in $\PGL_2(K)$.  Given the topology on $\proj_K^1$ generated by the open discs in $K$ as well as the subsets of the form $\{z \in K \ | \ v(z) > r\} \cup \{\infty\} \subset \proj_K^1$, it is straightforward to see that each element of $\PGL_2(K)$ acts as a self-homeomorphism on $\proj_K^1$.  In fact, the image of any closed (resp. open) disc in $K$ under the action of an element in $\PGL_2(K)$ is either a closed (resp. open) disc in $K$ or is the complement in $\proj_K^1$ of an open (resp. closed) disc in $K$.

Given any automorphism of $\proj_K^1$, there is a natural way to define its action on $\Berk$: see \cite[\S1.5, \S2.1]{baker2008introduction}, \cite[\S2.3]{baker2010potential}, \cite[\S7.1]{benedetto2019dynamics}, or \cite[\S II.1.3]{poineau2020berkovich}.  We describe concretely how each automorphism in $\PGL_2(K)$ acts on the points of $\Berk$ as well the properties of $\Berk$ that it respects.

\begin{prop} \label{prop pgl2 action on berk}

Any automorphism $\gamma \in \PGL_2(K)$ acts on $\Berk$ by sending points of Type I to points of Type I via its usual action on the points of $\proj_{\cc_K}^1$ and by sending points of Type II (resp. III) to points of Type II (resp. III) in the following manner.  If $\eta_D \in \Berk$ is a point of Type II or III corresponding to a (closed) disc $D \subset \cc_K$, we have $\gamma(\eta_D) = \eta_{E}$, where $E = \gamma(D)$ if $\infty \notin \gamma(D)$ and $E$ is the smallest closed disc containing $K \smallsetminus \gamma(D)$ if $\infty \in \gamma(D)$.

Each automorphism in $\PGL_2(K)$ acts as a metric-preserving self-homeomorphism on $\Berk$.

\end{prop}

\begin{proof}

The claims about how an automorphism $\gamma \in \PGL_2(K)$ acts on points of $\Berk$ can be deduced from \cite[Proposition 7.6, Theorem 7.12]{benedetto2019dynamics} is more or less proved in \cite[\S II.1.3]{poineau2020berkovich} (although the formula for $E$ in the case that $\infty \in \gamma(D)$ is not explicitly stated there).  The fact that $\gamma$ respects the metric of $\Berk$ is \cite[Proposition 2.30]{baker2010potential}.
\end{proof}

Eigenspaces of a matrix in $\GL_2(K)$ which represents an element of $\PGL_2(K)$ correspond to fixed points of $\proj_{\cc_K}^1$, and as a result, each nontrivial element of $\PGL_2(K)$ fixes exactly $1$ or exactly $2$ points in $\proj_{\cc_K}^1$.  We recall the following classification of elements of $\PGL_2(K)$.

\begin{dfn} \label{dfn loxodromic}

Given an element $\gamma \in \PGL_2(K)$ and a lifting to an element $\tilde{\gamma} \in \GL_2(K)$, we say that $\gamma$ is \emph{parabolic} if $\tilde{\gamma}$ has only $1$ eigenvalue (equivalently, $\gamma$ fixes exactly $1$ point of $\proj_{\cc_K}^1$); we say that $\gamma$ is \emph{elliptic} if $\tilde{\gamma}$ has $2$ distinct eigenvalues of the same valuation; and we say that $\gamma$ is \emph{loxodromic} if $\tilde{\gamma}$ has $2$ eigenvalues with distinct valuations.

\end{dfn}

We note that since the ratio of eigenvalues of an element of $\GL_2(K)$ is preserved by multiplication by a scalar matrix, the above properties are well defined for an element of $\PGL_2(K)$.  We now describe the action of an automorphism in $\PGL_2(K)$ of each of the above three types on $\Berk$.

\begin{prop} \label{prop loxodromic action on berk}

Let $\gamma \in \PGL_2(K)$ be an automorphism.  Given distinct points $a, b \in \proj_{\cc_K}^1$, we denote the axis connecting the corresponding points $\eta_a, \eta_b \in \Berk$ by $\Lambda_{a, b}$.

\begin{enumerate}[(a)]

\item Suppose that $\gamma$ is parabolic, and write $z \in \proj_{\cc_K}^1$ for its unique fixed point.  Then there is a point $\eta$ of Type II such that $\gamma$ fixes each point in $[\eta_z, \eta] \subset \Berk$.

\item Suppose that $\gamma$ is elliptic or loxodromic.  Here and below, write $a, b \in \proj_{\cc_K}^1$ for its fixed points, and write $u_\gamma \in \cc_K^\times$ for the ratio between the eigenvalues of a representative of $\gamma$ in $\GL_2(K)$, ordered so that $u_\gamma$ is integral.  Then we have $\gamma(\Lambda_{a, b}) = \Lambda_{a, b}$ and $\delta(\eta, \gamma(\eta)) = v(u_\gamma) \geq 0$ for any point $\eta \in \Lambda_{a, b} \smallsetminus \{\eta_a, \eta_b\}$.  If $\gamma$ is loxodromic (so that $v(u_\gamma) > 0$), it acts on each point in $\Lambda_{a, b} \smallsetminus \{\eta_a, \eta_b\}$ by moving it a distance of $v(u_\gamma)$ in a fixed direction on the axis.

\item If $\gamma$ is elliptic, the set of fixed points of $\gamma$ in $\Berk$ coincides with $B(\Lambda_{a, b}, v(u_\gamma - 1)) \subset \Berk$.

\item If $\gamma$ is elliptic (resp. loxodromic) and $\eta \in \Berk$ is any point of Type II or III, let $\xi \in \Berk$ be the closest point in $B(\Lambda_{a, b}, v(u_\gamma - 1))$ (resp. on $\Lambda_{a, b}$) to $\eta$.  The path $[\eta, \gamma(\eta)] \subset \Berk$ is the concatination of the paths $[\eta, \xi]$, $[\xi, \gamma(\xi)]$, and $[\gamma(\xi), \gamma(\eta)]$ (in the elliptic case the path $[\xi, \gamma(\xi)]$ contains just the point $\xi = \gamma(\xi) \in B(\Lambda_{a, b}, v(u_\gamma - 1))$); in other words, there is no backtracking on this concatenation of paths.  Moreover, the automorphism $\gamma$ maps the path $[\eta, \xi]$ homeomorphically onto the path $[\gamma(\eta), \gamma(\xi)]$.

\end{enumerate}

\end{prop}

\begin{proof}

Let $\gamma \in \PGL_2(K)$ be a parabolic element.  After applying an appropriate automorphism to $\Berk$ and conjugating $\gamma$ by that automorphism, we may assume that $\gamma$ is of the form $z \mapsto z + b$ for some $b \in K$.  Then it is clear that the only point of Type I fixed by $\gamma$ is $\eta_\infty \in \Berk$, and that $\gamma$ fixes $\eta_D \in \Berk$ where $D \subset \cc_K$ is a closed disc if and only if we have $v(b) \geq d(D)$.  Part (a) follows.

Now let $\gamma \in \PGL_2(K)$ be an elliptic or a loxodromic element.  Then there is an automorphism of $\Berk$ which moves $a$ to $0$ and $b$ to $\infty$ and which diagonalizes $\gamma$ in such a way that we may assume, after applying this automorphism, that $\gamma$ acts as $z \mapsto u_\gamma z$.  Then, given any closed disc $D \subset \cc_K$, we have $\gamma(\eta_D) = \eta_D$ if and only if we have $\gamma(D) = u_\gamma(D) = D$.  In the elliptic case, we have $v(u_\gamma) = 0$, and then it is easy to see that this in turn is equivalent to saying that $v(u_\gamma - 1) + v(z) = v((u_\gamma - 1)z) \geq d(D)$ for any element $z \in D$ of minimal valuation.  This is the same as saying that the closed disc $D' \subset \cc_K$ containing $D$ with radius $d(D') = d(D) - v(u_\gamma - 1)$ contains $0$, which means that $\delta(\eta_{D'}, \eta_D) \leq v(u - 1)$.  This proves part (c).  Meanwhile, in the loxodromic case, the axis from $\infty$ to $0$ is given by $\Lambda_{\infty, 0} = \{\eta_\infty\} \cup \{\eta_D \ | \ 0 \in D\} \cup \{\eta_0\}$, and the fact that $\gamma$ moves each point a distance of $v(u_\gamma) > 0$ along this path in the direction of $0$ follows from our above description of open neighborhoods of $\eta_0$ and $\eta_\infty$ and proves part (b), while in the elliptic case we have $v(u_\gamma) = 0$ and part (b) follows trivially.  (See \cite[\S II.1.4]{poineau2020berkovich} for more details on the loxodromic case; see also \cite[Remark II.1.12]{poineau2020berkovich} with regard to the parabolic and elliptic cases.)

To prove part (d), suppose that $\gamma$ is elliptic (resp. loxodromic) and assume the notation in the desired statement, so that the closest point in $B(\Lambda_{a, b}, v(u_\gamma - 1))$ (resp. $\Lambda_{a, b}$) to $\gamma(\eta)$ is $\gamma(\xi)$ as the action of the automorphism $\gamma$ on $\Berk$ is metric-preserving by \Cref{prop pgl2 action on berk}.  Now given any point $\eta' \in [\eta, \xi]$, we have by the metric-preserving property that $\delta(\gamma(\eta'), \gamma(\eta)) = \delta(\eta', \eta)$ and $\delta(\gamma(\eta'), \gamma(\xi)) = \delta(\eta', \xi)$; by the uniqueness of (non-backtracking) paths (see \Cref{rmk paths in berk}), we must have $\gamma(\eta') \in [\gamma(\eta), \gamma(\xi)]$.  The fact that $\gamma$ maps $[\eta, \xi]$ homeomorphically onto $[\gamma(\eta), \gamma(\xi)]$ immediately follows.

In the case that $\gamma$ is elliptic, we have $\gamma(\xi) = \xi$ and that $\gamma$ does not fix any point other than $\xi$ on the paths $[\eta, \xi]$ and $[\gamma(\eta), \gamma(\xi) = \xi]$ since the intersection of each path with $B(\Lambda_{a, b}, v(u_\gamma - 1))$ coincides with $\{\xi\}$.  We now see from the uniqueness of non-backtracking paths that the intersection $[\eta, \xi] \cap [\gamma(\eta), \xi]$ must coincide with $\{\xi\}$.  Thus, the concatination of the paths $[\eta, \xi]$, $[\xi, \gamma(\xi)] = \{\xi\}$, and $[\gamma(\xi), \gamma(\eta)]$ has no backtracking, and part (d) is proved in this case.  The statement of (d) for a loxodromic automorphism $\gamma$ follows from a similar exercise using \Cref{prop pgl2 action on berk}.
\end{proof}

\begin{cor} \label{cor loxodromic doesn't fix points}

A nontrivial automorphism $\gamma \in \PGL_2(K)$ is loxodromic if and only if it does not fix any point in $\Berk$ of Type II or III.

\end{cor}

\begin{proof}

This follows immediately from \Cref{prop loxodromic action on berk}(a)(b)(c)(d).
\end{proof}

As we will encounter elements of $\PGL_2(K)$ of order $p$ in the next subsection, the following proposition will be useful.

\begin{prop} \label{prop order-p elliptic}

Let $s \in \PGL_2(K)$ be an element of (finite) order $n$.

\begin{enumerate}[(a)]

\item The automorphism $s$ is elliptic and therefore fixes exactly $2$ points $a, b \in \proj_{\cc_K}^1$.  Conversely, any other order-$n$ element of $\PGL_2(K)$ which fixes the points $a, b \in \proj_K^1$ is a power of $s$.

\item If $n = p$ is prime, then the subset of $\Berk$ fixed by $s$ coincides with the tubular neighborhood $B(\Lambda_{a, b}, \frac{v(p)}{p - 1})$, where $\Lambda_{a, b}$ is the axis $[\eta_a, \eta_b] \subset \Berk$.  In particular, if the residue characteristic of $K$ is not $p$, the subset of $\Berk$ fixed by $s$ coincides with the axis $\Lambda_{a, b}$.

\end{enumerate}

\end{prop}

\begin{proof}

We see from the proof of \Cref{prop loxodromic action on berk}(a)(b)(c) that $\gamma$ must be elliptic and that we may assume that $\gamma$ acts as $z \mapsto u_\gamma z$.  Since $\gamma$ is of order $n$, we must have $u_\gamma^n = 1$, and therefore the set of all automorphisms $\gamma$ of order dividing $n$ which fix any particular two points of $\proj_K^1$ forms a cyclic subgroup of $\PGL_2(K)$ of order $n$, proving part (a).

If $n = p$ is prime, we have that $u_\gamma$ is a primitive $p$th root of unity, and the statement of part (b) follows from \Cref{prop loxodromic action on berk}(b) combined with the well-known fact that $v(\zeta_p - 1) = \frac{v(p)}{p - 1}$.
\end{proof}

\subsection{Schottky and Whittaker groups} \label{sec 2 schottky}

We begin this subsection by setting up the notation we need to discuss the action of particular subgroups of $\PGL_2(K)$ on $\Berk$.

\begin{dfn} \label{dfn limit points}

Given a subgroup $\Gamma < \PGL_2(K)$, the subset of \emph{limit points} of $\Gamma$ consists of all points $z \in \proj_K^1$ such that there exists a point $z_0 \in \proj_K^1$ and a sequence $\{\gamma_i\}_{i \geq 1} \subset \Gamma$ of distinct elements such that $\displaystyle \lim_{i \to \infty} \gamma_i(z_0) = z$ (where the limit is defined with respect to the metric on $\proj_{\cc_K}^1$ induced by the valuation $v : \cc_K \to \rr$).  We write $\mathcal{L}_\Gamma \subseteq \proj_K^1$ and $\Omega_\Gamma := \proj_K^1 \smallsetminus \mathcal{L}_\Gamma \subset \proj_K^1$ for the subset of limit points of $\Gamma$ and its complement respectively.  

For convenience, we will usually suppress the subscript $\Gamma$ in the notation $\Omega_\Gamma$ and simply write $\Omega \subset \proj_K^1$ for the set of non-limit points of a subgroup $\Gamma < \PGL_2(K)$ that we are working with.

\end{dfn}

\begin{rmk} \label{rmk non-loxodromic limit points}

If a subgroup $\Gamma < \PGL_2(K)$ contains a non-loxodromic element of infinite order, then we have $\mathcal{L}_\Gamma = \proj_K^1$; in other words, every point on the projective line is a limit point.  Indeed, if $\gamma \in \PGL_2(K)$ is non-loxodromic and has infinite order, after applying an appropriate automorphism as in the proof of \Cref{prop loxodromic action on berk}(a)(b), we may assume that $\gamma$ is of the form $z \mapsto z + b$ for $b \in K$ or $z \mapsto uz$ for $u \in K^\times$ of infinite order with $v(u) = 0$.  Then given any $z \in \proj_K^1$, it is elementary to show that there is a subsequence of $\{\gamma^i(z)\}_{i \geq 1}$ which converges to $z$, and thus we have $z \in \mathcal{L}_\Gamma$.

\end{rmk}

There are several equivalent definitions of a Schottky group (in the non-archimedean setting) that can be found in the literature.  For our purposes, the most directly useful one will be the following definition.

\begin{dfn} \label{dfn schottky}

A finitely-generated subgroup $\Gamma < \PGL_2(K)$ is said to be \emph{Schottky} if every nontrivial element of $\Gamma$ is loxodromic.

\end{dfn}

It is shown in \cite[\S I.4.1.3]{gerritzen2006schottky} using \textit{good fundamental domains} that the set of limit points of a Schottky group $\Gamma$ is strictly contained in $\proj_K^1$, and so we have $\Omega_\Gamma \neq \varnothing$ in this case (in other words, the action of $\Gamma$ on $\proj_K^1$ is \textit{discontinuous}).  The fact that $\Gamma$ acts on a nonempty subset of $\proj_K^1$ discontinuously is in fact what enables us to speak of a well-behaved quotient with respect to the action of $\Gamma$.  It is moreover shown in \cite[\S I.3]{gerritzen2006schottky} that a Schottky group is free.  We will independently recover the latter fact in the superelliptic case that we are interested in: see \Cref{rmk free} below.

As was summarized in \S\ref{sec 1 superelliptic Schottky}, given a Schottky group $\Gamma < \PGL_2(K)$ generated by $g$ elements, the corresponding quotient $\Omega / \Gamma$ can be realized as a split degenerate curve over $K$ of genus $g$.  Moreover, this curve is superelliptic and a degree-$p$ cover of $\proj_K^1$ if and only if there is a larger subgroup $\Gamma_0 < \PGL_2(K)$ generated by $g + 1$ elements $s_0, \dots, s_g$ each of order $p$, containing $\Gamma$ as a normal subgroup with index $p$, and such that $\Gamma$ is generated by the elements $\gamma_{i, j} := s_0^{j - 1} s_i s_0^{-j}$ for $0 \leq i \leq g$ and $1 \leq j \leq p - 1$ as in (\ref{eq Gamma}); see \cite[\S2]{van1982galois} for proofs of these facts and more details.  We call such a subgroup $\Gamma_0 < \PGL_2(K)$ a $p$-\emph{Whittaker group}.

\begin{rmk} \label{rmk Whittaker}

To the best of the author's knowledge, \textit{Whittaker groups} are defined only in the hyperelliptic case (\textit{i.e.} when $p = 2$) and then are defined in more than one way in the literature.  In some of the literature, such as \cite{gerritzen2006schottky}, a Whittaker group is defined in this situation to be the Schottky group $\Gamma$ itself rather than the group $\Gamma_0$.  Here we are choosing to more closely follow \cite[Definition 1.5]{van1983non}, which defines $\Gamma_0$ as a Whittaker group in the $p = 2$ situation; we extend the definition to include the superelliptic case where $p$ may be any prime by instead using the term $p$-\emph{Whittaker group}.

\end{rmk}

\subsection{Clusters and sets which are clustered in pairs} \label{sec 2 clustered in pairs}

As we have $\zeta_p \in K$, the fixed points in $\proj_{\cc_K}^1$ of each generator $s_i$ of a $p$-Whittaker group actually lie in $\proj_K^1$.  Conversely, setting $S \subset \proj_K^1$ be a $(2g + 2)$-element subset for some $g \geq 1$ with elements labeled as $a_0, b_0, \dots, a_g, b_g \in \proj_K^1$, with associated elements $s_i \in \PGL_2(K)$ and subgroups $\Gamma \lhd \Gamma_0 < \PGL_2(K)$ as in \Cref{dfn associated to S}, our main goal in this paper is to find a way to detect directly from $S$ whether the group $\Gamma_0$ is a $p$-Whittaker group.  This subsection focuses on a necessary (but generally not sufficient) property that must hold for $S$ in order for this to be the case.

From now on, again following \Cref{dfn associated to S}, we write $\Lambda_{(i)}$ for the axis $[\eta_{a_i}, \eta_{b_i}] \in \Berk$ connecting the points of Type I corresponding to the fixed points $a_i, b_i$.  For $0 \leq i \leq g$, \Cref{prop loxodromic action on berk}(b) and \Cref{prop order-p elliptic} show that the subset of $\Berk$ which is fixed pointwise under $s_i$ coincides with 
\begin{equation*}
\hat{\Lambda}_{(i)} := B(\Lambda_{(i)}, \tfrac{v(p)}{p - 1}) \subset \Berk.
\end{equation*}
Note that each subspace $\hat{\Lambda}_{(i)} \subset \Berk$ of points fixed by $s_i$ contains the corresponding axis $\Lambda_{(i)}$, and that the containment is an equality if and only if the residue characteristic of $K$ is not $p$.

\Cref{prop clustered in pairs} below expresses a condition on the set $\{a_0, b_0\}, \dots, \{a_g, b_g\}$ of pairs of fixed points of generators $s_i$ of $\Gamma_0$ which is necessary in order for the subgroup generated by the elements $\gamma_{i, j} = s_0^{j - 1} s_i s_0^{-j}$ to be Schottky.  We note in \Cref{rmk converse} below that a stronger version of this proposition exists which relies on the freeness of the Schottky group, along with a converse statement, but the version we present in this subsection will be sufficient for obtaining the results we want.

\begin{prop} \label{prop clustered in pairs}

Let $S \subset \proj_K^1$ be a $p$-superelliptic subset with associated $p$-Whittaker group $\Gamma_0 < \PGL_2(K)$ and subspaces $\hat{\Lambda}_{(i)} \subset \Berk$ as above.  Suppose that we have $s_j^n(\hat{\Lambda}_{(i)}) \cap \hat{\Lambda}_{(l)} \neq \varnothing$ for some indices $i, j, l \in \{0, \dots, g\}$ with $i \neq j$ and some $n \in \zz$.  Then we have $i = l$ and $n \in p\zz$.  In particular, we have $\hat{\Lambda}_{(i)} \cap \hat{\Lambda}_{(l)} = \varnothing$ for $i \neq l$.

\end{prop}

\begin{proof}

It is a straightforward exercise to observe from the topology of $\Berk$ that any non-disjoint paths must intersect at a point of Type II or III.  Suppose that for such indices $i \neq j, l$ and an integer $n \in \zz$, there is a point $\eta$ of Type II or III lying in the intersection $s_j^n(\hat{\Lambda}_{(i)}) \cap \hat{\Lambda}_{(l)}$.  We first show that this implies that $i = l$.  Since the elements $s_j^n s_i s_j^{-n}$ and $s_l$ respectively fix $s_j^n(\hat{\Lambda}_{(i)})$ and $\hat{\Lambda}_{(l)}$ pointwise, the product $s_j^n s_i s_j^{-n} s_l^{-1} \in \Gamma$ fixes $\eta$.  Since no nontrivial element of a Schottky group fixes a point of Type II or III (by \Cref{prop loxodromic action on berk}), we have $s_j^n s_i s_j^{-n} s_l^{-1} = 1$, or equivalently, that $s_l = s_j^n s_i s_j^{-n}$.  If $i \neq l$, then this implies that $j \neq l$ also, and then the $p$-Whittaker group $\Gamma_0$ is generated by the $g$-element set $\{s_0, \dots, s_g\} \smallsetminus \{s_l\}$, which contradicts \Cref{dfn superelliptic set}(ii).  We therefore have $i = l$; now applying the result just shown to the case when $n = 0$ and the indices $i$ and $l$ are replaced by an arbitrary pair of indices, we get that the subspaces $\hat{\Lambda}_{(0)}, \dots, \hat{\Lambda}_{(g)} \subset \Berk$ are mutually disjoint.

We still need to show that $s_j^n(\hat{\Lambda}_{(i)}) \cap \hat{\Lambda}_{(l)} = s_j^n(\hat{\Lambda}_{(i)}) \cap \hat{\Lambda}_{(i)} \neq \varnothing$ implies $n \in p\zz$.  Assume that $n \notin p\zz$ and choose any point $\eta \in \Lambda_{(i)}$; our goal is to show that $s_j^n(\eta) \notin \Lambda_{(i)}$ for any $j \neq i$.  From our observation in the end of the last paragraph, we have $\hat{\Lambda}_{(i)} \cap \hat{\Lambda}_{(j)} = \varnothing$ and therefore, the closest point $\xi$ in $\hat{\Lambda}_{(j)}$ to $\eta$ does not lie in $\Lambda_{(i)}$.  Meanwhile, it follows from \Cref{prop loxodromic action on berk}(d) that the path $[\eta, \xi] \cup [\xi, s_j^n(\eta)]$ is non-backtracking.  Now the desired fact that $s_j^n(\eta) \notin \Lambda_{(i)}$ follows easily from the uniqueness of non-backtracking paths as discussed in \Cref{rmk paths in berk}.
\end{proof}

\begin{dfn} \label{dfn clustered in pairs}

Let $A$ be a finite multiset of positive even cardinality consisting of points in $\proj_K^1$.  We say that $A$ is \emph{clustered in $r$-separated pairs} for some $r \in \rr_{\geq 0}$ if there is a labeling $a_0, b_0, \dots, a_g, b_g$ of the elements of $A$ such that we have 
\begin{enumerate}[(i)]
\item $a_i \neq b_i$ for $1 \leq i \leq g$ and 
\item $B(\Lambda_{(i)}, r) \cap B(\Lambda_{(j)}, r) = \varnothing$ for $i \neq j$.
\end{enumerate}

If $A$ is clustered in $0$-separated pairs, we say more simply that $A$ is \emph{clustered in pairs}.

In the context of using this terminology, the \emph{pairs} that $A$ is clustered in are the $2$-element sets $\{a_i, b_i\}$.

\end{dfn}

\begin{rmk} \label{rmk clustered in pairs}

We make the following crucial observations concerning clustering in pairs.

\begin{enumerate}[(a)]

\item An important consequence of \Cref{prop clustered in pairs} is that if an even-cardinality subset $S \subset \proj_K^1$ is superelliptic with respect to a labeling of its elements as $a_0, b_0, \dots, a_g, b_g$, then it must be clustered in $\frac{v(p)}{p - 1}$-separated pairs, and the pairs are the subsets $\{a_i, b_i\}$.

\item It is not difficult to verify, particularly from the definition outlined in \Cref{rmk clustered in pairs alternate definition} below, that a subset of $\proj_{\cc_K}^1$ which is clustered in pairs under a particular labeling of its elements cannot be clustered in pairs under any other labeling; in other words, a set can be clustered in at most one set of pairs.  This observation, together with part (a), imply that there is at most one unique way to partition a set $S \subset \proj_K^1$ into $2$-element subsets $\{a_i, b_i\}$ so that the order-$p$ elements $s_i$ fixing these pairs generate a $p$-Whittaker group.

\end{enumerate}

\end{rmk}

It is often useful and more intuitive to characterize the property of being clustered in ($r$-separated) pairs in terms of the \emph{cluster data} of the set $S$ (which in turn can be computed directly and easily from the distances between the points in $S \smallsetminus \{\infty\}$ with respect to the valuation on $K$).  We define the notion of \emph{clusters} below, following \cite[Definition 1.1]{dokchitser2022arithmetic}.

\begin{dfn} \label{dfn cluster}

Let $A \subset \proj_K^1$ be a finite subset.  A subset $\mathfrak{s} \subseteq A$ is called a \emph{cluster} (of $A$) if there is some disc $D \subset K$ such that $\mathfrak{s} = A \cap D$.  The \emph{depth} of a cluster $\mathfrak{s}$ is the integer 
\begin{equation}
d(\mathfrak{s}) := \min_{z, z' \in \mathfrak{s}} v(z - z').
\end{equation}
The data of all clusters of a finite subset $A \subset \proj_K^1$ along with each of their depths is called the \emph{cluster data} of $A$.

\end{dfn}

The use of the same notation for the depth of a disc as for the (logarithmic) radius of a disc (given at the top of this section) is deliberate: if $D_{\mathfrak{s}} \subset \cc_K$ is the minimal disc containing a cluster $\mathfrak{s}$ of a finite subset $A \subset K$, then it is immediate from the definitions that we have $d(D_{\mathfrak{s}}) = d(\mathfrak{s})$.

\begin{rmk} \label{rmk clustered in pairs alternate definition}

As mentioned above, it will be useful in some contexts to rephrase \Cref{dfn clustered in pairs} in the language of clusters, which is indeed what inspired the phrasing in the terminology.  The downside of this is that the definitions themselves are then less elegant to state and sometimes more difficult to manipulate in arguments.

To avoid making the following definitions too cumbersome, we assume that $A \subset \proj_{\cc_K}^1$ is a subset (rather than just a multiset).  Define an equivalence relation $\sim$ on $A$ as follows: given two points $z, z' \in A$, we write $z \sim z'$ if $z$ and $z'$ lie in the exact same even-cardinality clusters of $A$, \textit{i.e.} if for every even-cardinality cluster $\mathfrak{s} \subseteq A$ we have either $z, z' \in \mathfrak{s}$ or $z, z' \notin \mathfrak{s}$.

Then we may define \emph{clustered in pairs} by saying that a non-empty even-cardinality subset $A \subset \proj_K^1$ is clustered in pairs (where the pairs are particular subsets $\{a_i, b_i\} \subset A$ for $0 \leq i \leq g$) if the equivalence classes in $A$ under $\sim$ all have cardinality $2$ (and coincide with the subsets $\{a_i, b_i\}$).

To define \emph{clustered in $r$-separated pairs} in this language is slightly more complicated.  Given any cluster $\mathfrak{s} \subset A$, write $\mathfrak{s}'$ (resp. $\mathfrak{s}^\sim$) for the smallest cluster properly containing $\mathfrak{s}$ (resp. properly containing $\mathfrak{s}$ and which is not itself the disjoint union of $\geq 2$ even-cardinality sub-clusters), if such a cluster properly containing $\mathfrak{s}$ exists.  Then a set $A$ is clustered in $r$-separated pairs for some $r \in \rr_{\geq 0}$ if it is clustered in pairs and if the following two conditions hold, for any even-cardinality clusters $\mathfrak{c}, \mathfrak{c}_1, \mathfrak{c}_2$ which are themselves not the disjoint union of $\geq 2$ even-cardinality clusters:
\begin{enumerate}[(i)]
\item if $\mathfrak{c}^\sim$ is defined, we have $d(\mathfrak{c}) - d(\mathfrak{c}^\sim) > 2r$; and 
\item if $\mathfrak{c}_1', \mathfrak{c}_2'$ are defined and $\mathfrak{s} := \mathfrak{c}_1' = \mathfrak{c}_2'$, we have $d(\mathfrak{c}_1) + d(\mathfrak{c}_2) - 2d(\mathfrak{s}) > 2r$.
\end{enumerate}

\end{rmk}

In the special case that $\#S = 4$ and $p = 2$ (in which case if $S$ is superelliptic we produce a uniformization of an elliptic curve), it is proved in \cite[\S IX.2.5]{gerritzen2006schottky}, under the assumption that the residue characteristic of $K$ is not $2$ (so that $\frac{v(p)}{p - 1} = 0$), that the converse to \Cref{rmk clustered in pairs}(a) holds: being clustered in pairs is a sufficient condition for a set $S$ to be superelliptic.  We will see in \Cref{prop g = 1 converse} below that our above results prove this converse for any $p$ and with the residue characteristic condition dropped.  Such a converse no longer holds in general, however, as soon as the cardinality of $S$ exceeds $4$, as the below example shows.

\begin{ex} \label{ex clustered in pairs not sufficient}

Let $K = \qq_5$, equipped with the usual $5$-adic valuation $v : \qq_5^\times \to \zz$, and consider the subset $S := \{7, 12, 0, 5, 1, \infty\} \subset \proj_K^1$.  It is easy to see that $S$ is clustered in ($\frac{v(2)}{2 - 1} = 0$-separated) pairs, which are $\{a_0 := 7, b_0 := 12\}, \{a_1 := 0, b_1 := 5\}, \{a_2 := 1, b_2 := \infty\}$.  We now show that the set $S$ is not $2$-superelliptic.  Letting $s_i \in \PGL_2(K)$ be (as usual) the unique order-$2$ element which fixes $a_i$ and $b_i$ for $i = 0, 1, 2$, we may write matrix representations for these fractional linear transformations as 
\begin{equation} \label{eq clustered in pairs not sufficient}
s_0 = \begin{bmatrix} 19 & -168 \\ 2 & -19 \end{bmatrix}, \ s_1 = \begin{bmatrix} 5 & 0 \\ 2 & -5 \end{bmatrix}, \ s_2 = \begin{bmatrix} -1 & 2 \\ 0 & -1 \end{bmatrix}.
\end{equation}
As in (\ref{eq Gamma}), we have $\Gamma_0 = \langle s_0, s_1, s_2 \rangle$ and $\Gamma = \langle s_1 s_0, s_2 s_0 \rangle$, and therefore we have $s_1 s_2 s_0 s_2 = (s_1 s_0)(s_2 s_0)^{-2} \in \Gamma$.  One now computes directly from (\ref{eq clustered in pairs not sufficient}) that a representative matrix for $s_1 s_2 s_0 s_2$ has characteristic polynomial $T^2 - 350T + 625$.  By examining the Newton polygon of this characteristic polynomial, we see that both of its roots have valuation $2$.  The product $s_1 s_2 s_0 s_2 \in \Gamma$ is therefore by definition not loxodromic, and so $S$ is not $2$-superelliptic despite being clustered in pairs.

\end{ex}

\section{Group actions on convex hulls of sets of fixed points} \label{sec 3}

For the rest of the paper, we will make use of the following notation.

\begin{dfn} \label{dfn D_i}

Let $A \subset \proj_K^1$ be a subset which is clustered in the pairs $\{a_0, b_0\}, \dots, \{a_g, b_g\}$.  Then for $0 \leq i \leq g$, if $\infty \notin \{a_i, b_i\}$ (resp. if $\infty \in \{a_i, b_i\}$), we define $D_i \subset \cc_K$ to be the minimal disc containing both points $a_i, b_i \in K$ (resp. containing $A \smallsetminus \{\infty\}$).  Write $v_i = \eta_{D_i} \in \Berk$ for the points of Type II corresponding to these discs.

\end{dfn}

\subsection{Convex hulls and reduced convex hulls of finite sets} \label{sec 3 convex hull}

We begin by defining two subspaces of $\Berk$ associated to a given finite subset of $\proj_K^1$ (which in practice will be the set $S$ of fixed points of a $p$-Whittaker group).

\begin{dfn} \label{dfn reduced convex hull}

Given a finite subset $A \subset \proj_K^1$ viewed as a subset of point of Type I in $\Berk$, we define the \emph{convex hull} of $A$ to be the smallest connected subset $\Sigma_A \subset \Berk$ containing $A$ (which is well defined because any two endpoints has a unique path in the space $\Berk$).

If $A$ has even cardinality, let $\Sigma_{A, 1} \subset \Sigma_A$ be the subset consisting of the points $\eta \in \Sigma_A \smallsetminus A$ such that the $2$ connected components of $\Sigma_A \smallsetminus \{\eta\}$ each have odd-cardinality intersection with $A$.  We define the \emph{reduced convex hull} $\Sigma_{A, 0}$ of $A$ to be the closure of the interior of $\Sigma_A \smallsetminus \Sigma_{A, 1}$ as a subspace of $\Sigma_A$.

If $A$ is instead a multiset consisting of elements of $\proj_K^1$ with underlying set $\underline{A}$, we write $\Sigma_A$ for $\Sigma_{\underline{A}}$ and (if $\underline{A}$ has even cardinality) $\Sigma_{A, 0}$ for $\Sigma_{\underline{A}, 0}$.

\end{dfn}

We recall that a \emph{real tree} can be viewed topologically as a space in which each point has an open neighborhood homeomorphic to an open interval in $\rr$, except for a finite number of points, which we call the \emph{natural vertices}, whose open neighborhoods are not homeomorphic to an open interval but either are a half-closed interval or contain a star shape centered at the vertex.  A real tree has the structure of a metric graph whose vertices are the natural vertices described above; we may also enhance such a space by specifying other points as vertices to be included in the data of its metric graph structure (any such ``extra'' vertex which is not a natural vertex must have valency $2$), noting that this operation does not affect whether or not the metric graph is finite.  In the situation that we are interested in, we are able to define such ``extra'' vertices as follows.

\begin{dfn} \label{dfn distinguished vertices}

Let $A \subset \proj_K^1$ be a $(2g + 2)$-element subset which is clustered in pairs $\{a_0, b_0\}, \dots, \{a_g, b_g\}$, and let the axes $\Lambda_{(0)}, \dots, \Lambda_{(g)} \subset \Berk$ be as in \Cref{dfn associated to S}.  We designate the \emph{distinguished vertices} of the reduced convex hull $\Sigma_{A, 0}$ to be the points lying in $\Sigma_{A, 0} \cap (\Lambda_{(0)} \cup \dots \cup \Lambda_{(g)})$.

\end{dfn}

We now establish some properties of reduced convex hulls regarding their structures as metric graphs.

\begin{prop} \label{prop reduced convex hull}

Suppose that an even-cardinality subset $A \subset \proj_K^1$ is clustered in $r$-separated pairs labeled $\{a_0, b_0\}, \dots, \{a_g, b_g\}$ for some $r \geq 0$ and $g \geq 1$.  The reduced convex hull $\Sigma_{A, 0}$ is a real tree which, enhanced by specifying the distinguished vertices as defined in \Cref{dfn reduced convex hull} as vertices, has the structure of a finite metric graph.

Let $o$ be the number of odd-cardinality clusters of $A$ whose cardinality is not $1$ or $\#A - 1 = 2g + 1$.

\begin{enumerate}[(a)]

\item The subspace $\Sigma_{A, 0} \subset \Sigma_A$ is the closure of the interior of $\Sigma_A \smallsetminus \{\Lambda_{(0)} \sqcup \dots \sqcup \Lambda_{(g)}\}$ as a subspace of $\Sigma_A$.

\item For $0 \leq i \leq g$, given a point $\eta \in \Sigma_A$, the closest point in $\Sigma_{A, 0}$ to $\eta$ lies in $\Lambda_{(i)}$ (resp. $\hat{\Lambda}_{(i)}$) if and only if we have $\eta \in \Lambda_{(i)}$ (resp. $\eta \in \hat{\Lambda}_{(i)}$).

\item All vertices of $\Sigma_{A, 0}$ with valency $\leq 2$ are distinguished.

\item For each index $i$, the set of distinguished vertices $\Sigma_{A, 0} \cap \Lambda_{(i)}$ consists of the points $\eta_D$ for all discs $D \subset \cc_K$ which minimally constains an odd-cardinality cluster $\mathfrak{s} \subseteq A$ such that $3 \leq \#\mathfrak{s} \leq 2g + 1$ and $\#(\mathfrak{s} \cap \{a_i, b_i\}) = 1$, along with the point $v_i = \eta_{D_i}$ unless we have $D_i \supset A \not\ni \infty$ and no cluster of cardinality $2g + 1$.  Each of these vertices lies in a separate connected component of $\Sigma_{A, 0}$.  The total number of distinguished vertices of $\Sigma_{A, 0}$ is $o + g + 1$.

\item The number of connected components of $\Sigma_{A, 0}$ is $o + 1$.

\item The distance between each pair of distinct distinguished vertices lying in the same component of $\Sigma_{A, 0}$ is a rational number in $v(K^\times)$ which is $> 2r$.

\item Let $\Sigma_0' \subseteq \Sigma_{A, 0}$ be a component of $\Sigma_{A, 0}$.  Then we have $\Sigma_0' = \Sigma_{A', 0}$, where $A'$ consists of the points $a_i, b_i \in A$ for each index $i$ such that $\Sigma_0' \cap \Lambda_{(i)} \neq \varnothing$.

\end{enumerate}

\end{prop}

\begin{proof}

We begin with the observation that for any index $i$, we may describe the points in $\Lambda_{(i)} \smallsetminus \{\eta_{a_i}, \eta_{b_i}\}$ as those points $\eta_D$ where $D \subset \cc_K$ such that either we have $\#(D \cap \{a_i, b_i\}) = 1$ or (in the case that $\infty \notin \{a_i, b_i\}$) we have $D = D_i$ (the smallest disc containing $\{a_i, b_i\}$).  We observe moreover that if $D \subset \cc_K$ is the smallest disc containing \textit{any} particular subset (cluster) of $A$, then there is some index $i$ such that $D$ is the minimal disc containing $\{a_i, b_i\}$ (resp. we have $\#(D \cap \{a_i, b_i\}) = 1$) if $\#(D \cap A)$ is even (resp. odd) and therefore the corresponding point $\eta_D$ lies in one of the axes $\Lambda_{(i)}$.  We will freely use this observations throughout the rest of the proof.

Each point $\eta \in \Sigma_A$ which separates $A$ into odd-cardinality subsets as in \Cref{dfn reduced convex hull} must lie in some axis $\Lambda_{(i)}$, while conversely, for each $i$, each point in $\Lambda_{(i)} \smallsetminus \{\eta_{a_i}, \eta_{b_i}, v_i = \eta_{D_i}\}$ separates $A$ into odd-cardinality subsets (as does the point $v_i$ if $\infty \in \{a_i, b_i\}$).  Since $\eta_{a_i}, \eta_{b_i}$ are isolated points in each axis $\Lambda_{(i)}$, as is $v_i$ if $\infty \notin \{a_i, b_i\}$, part (a) follows.

Now let $\eta \in \Sigma_A$ be any point.  If we have $\eta \in \Sigma_{A, 0}$, then we already get the conclusion of part (b), so let us assume that $\eta \in \Sigma_A \smallsetminus \Sigma_{A, 0}$.  Then it follows from part (a) that we have $\eta \in \Lambda_{(i)} \subseteq \hat{\Lambda}_{(i)}$ for a unique index $i$ (the uniqueness coming from the fact that the $\hat{\Lambda}_{(i)}$'s are mutually disjoint) and that the path from $\eta$ to any point on $\Sigma_{A, 0}$ must pass through a point in $\Lambda_{(i)} \cap \Sigma_{A, 0}$.  This implies part (b).

By construction of the graph structure on $\Sigma_{A, 0}$, all vertices which are not natural are distinguished.  A vertex of valency $2$ cannot be natural, since it has a neighborhood homeomorphic to an open subset of $\rr$, so it is distinguished.  Now let $v$ be a vertex of valency $\leq 1$.  Any point of Type II or III in $\Sigma_A$ clearly has a neighborhood homeomorphic to an open interval in $\rr$ or a star shape centered at $v$, by the fact that $\Sigma_A \subset \Berk$ is the smallest connected subspace containing a fixed set of points of Type I.  Therefore, any neighborhood of $v$ contains a point in $\Sigma_A \smallsetminus \Sigma_{A, 0}$.  Part (a) implies that such a point lies in one of the axes $\Lambda_{(i)}$, and thus, so does $v$ as it is a limit of such points; this makes $v$ a distinguished vertex by definition and proves part (c).

We now fix an index $i$ and set out to characterize all (distinguished) vertices lying in $\Sigma_{A, 0} \cap \Lambda_{(i)}$.  It follows from part (a) that each such point must be the nearest point on $\Lambda_{(i)}$ to $\Lambda_{(j)}$ for some index $j \neq i$.  Choose a point $\eta_D \in \Lambda_{(i)}$ such that the corresponding disc $D \subset \cc_K$ is not the minimal disc containing the cluster $D \cap A$.  Choose an index $j \neq i$ and a point $\eta_{D'} \subset \Lambda_{(j)}$ and consider the path $[\eta_D, \eta_{D'}] \subset \Sigma_A$.  Using the notation of \Cref{rmk paths in berk}, if the disc corresponding to $\eta_D \vee \eta_{D'}$ strictly contains $D$, then the sub-path $[\eta_D, \eta_D \vee \eta_{D'}] \subseteq [\eta_D, \eta_{D'}]$ has an interior, and there clearly exists a disc $D'' \supsetneq D$ satisfying $D'' \cap A = D \cap A$ and $\eta_{D''} \in [\eta_D, \eta_{D \vee D'}]$.  Then by our above description of points of $\Lambda_{(i)}$, we have $\eta_{D''} \in \Lambda_{(i)}$ and so the point $\eta_D$ is not the closest point in $\Lambda_{(i)}$ to $\Lambda_{(j)}$.  Similarly, if instead we have $\eta_D \vee \eta_{D'} = \eta_D$, then we get the inclusion $D' \subsetneq D$; from the fact that $D' \cap A \neq \varnothing$ (as $\eta_{D'} \in \Lambda_{(j)}$), it is easy to deduce the existence of a disc $D'' \subsetneq D$ satisfying $D'' \cap A = D \cap A$ and $\eta_{D''} \in [\eta_D, \eta_{D'}]$.  Then again we get $\eta_{D''} \in \Lambda_{(i)}$ and reach the same conclusion about the point $\eta_D$.

We have thus shown that any point $\eta_D \in \Sigma_{A, 0} \cap \Lambda_{(i)}$ satisfies that the corresponding disc $D \subset \cc_K$ is the smallest disc containing the cluster $D \cap A$.  Conversely, we now let $D \subset \cc_K$ be any disc such that $D$ is the smallest disc containing the cluster $\mathfrak{s} := D \cap A$.  Let us assume first that $\mathfrak{s}$ has even cardinality, so that we must have $D = D_i$ for some index $i$.  Assume further that we have $\mathfrak{s} \subsetneq A$.  Then there is a minimal cluster $\mathfrak{s}' \supsetneq \mathfrak{s}$, and it follows from \Cref{rmk clustered in pairs alternate definition} that we have $\#(\mathfrak{s}' \smallsetminus \mathfrak{s}) \in \{1, 2\}$.  Without loss of generality, we may assume that $a_j \in \mathfrak{s}' \smallsetminus \mathfrak{s}$ for some index $j \neq i$.  Let $D' \subset \cc_K$ be the smallest disc containing the cluster $\mathfrak{s}'$.  Then we have $\eta_{D'} \in \Lambda_{(j)}$, and the path $[\eta_D, \eta_{D'}] \subset \Sigma_{A, 0}$ consists of points of the form $\eta_{D''}$ with $D'' \supsetneq D$; this implies that $[\eta_D, \eta_{D'}] \cap \Lambda_{(i)} = \{\eta_D\}$.  Thus, the point $\eta_D$ is a limit point of $\Sigma_{A, 0} \smallsetminus \{\Lambda_{(0)} \sqcup \dots \sqcup \Lambda_{(g)}\}$ and by part (a) lies in $\Sigma_{A, 0}$.

Suppose on the other hand that we have $\infty \notin A$ and $\mathfrak{s} = A$, so that $D = D_i$.  Each point $\eta_{D'} \in (\Lambda_{(0)} \sqcup \dots \sqcup \Lambda_{(g)}) \smallsetminus \Lambda_{(i)}$ corresponds to a disc $D' \subsetneq D$ satisfying $D \cap A \neq \varnothing$, and the interior of the path $[\eta_D, \eta_{D'}] \subset \Sigma_A$ consists of points of the form $\eta_{D''}$ with $D' \subsetneq D'' \subsetneq D$.  If there is a cluster $\mathfrak{s}''$ of cardinality $\#A - 1$, then for sufficiently large such discs $D''$ we have $D'' \cap A = \mathfrak{s}''$; since we have $(A \smallsetminus \mathfrak{s}'' \cap \{a_i, b_i\}) = 1$, this implies that the corresponding point $\eta_{D''}$ sufficiently close to $\eta_D$ lies in $\Lambda_{(i)}$, and so $\eta_D$ is not the closest point of $\Lambda_{(i)}$ to any other axis $\Lambda_{(j)}$ and we have $\eta_D \notin \Sigma_{A, 0}$.  If instead there is no cluster of cardinality $\#A - 1$, then we have $D'' \cap \{a_i, b_i\} = \varnothing$ for every such $D''$, implying that $\eta_{D''} \notin \Lambda_{(i)}$; by a similar limit point argument to the one used above, we get $\eta_D \in \Sigma_{A, 0}$.

Finally, assume that $\mathfrak{s}$ has odd cardinality $\geq 3$.  Then we have $\#(\mathfrak{s} \cap \{a_i, b_i\}) = 1$, and one can easily deduce from \Cref{rmk clustered in pairs alternate definition} that there is a maximal subcluster $\mathfrak{s}' \subsetneq \mathfrak{s}$ with $\mathfrak{s}' \cap \{a_i, b_i\} = \varnothing$.  Letting $D'$ be the minimal cluster containing $\mathfrak{s}'$, we get $\eta_D \in \Lambda_{(j)}$ for some index $j \neq i$.  By arguments to those used above involving paths and limit points, we then get $\eta_D \in \Sigma_{A, 0}$.

We have shown that for each index $i$, the intersection $\Sigma_{A, 0} \cap \Lambda_{(i)}$ is a discrete set; now by uniqueness of (non-backtracking) paths (see \Cref{rmk paths in berk}) it follows that each point in this intersection lies in a different component of $\Sigma_{A, 0}$.  We are able to count $1$ distinguished vertex for each cluster of odd cardinality $\neq 1, 2g + 1$ as well as the $g + 1$ distinguished vertices $v_0', \dots, v_g'$, where $v_i'$ is the vertex corresponding to the smallest disc containing the (necessarily unique) cardinality-($2g + 1$) cluster if such a cluster exists and $v_i' = v_i$ otherwise.  This completes the proof of part (d).

Meanwhile, the above description of a distinguished vertex in each intersection $\Sigma_{A, 0} \cap \Lambda_{(i)}$ as the closest point in the axis $\Lambda_{(i)}$ to some other axis $\Lambda_{(j)}$ implies that the reduced convex hull $\Sigma_{A, 0}$ is contained in the smallest connected subset $\Sigma_A' \subset \Sigma_A$ containing all distinguished vertices.  Clearly the points of $\Sigma_A'$ are each of Type II or III and thus the distances between them are finite.  Part (a) implies that the subset $\Sigma_{A, 0} \subseteq \Sigma_A'$ is obtained by removing the interior of each path in $\Lambda_{(i)}$ between distinguished vertices in $\Sigma_{A, 0} \cap \Lambda_{(i)}$ for each $i$.  From unique path-connectness one now easily sees that $\Sigma_A'$ is a finite metric graph and so the same is true of $\Sigma_{A, 0}$, as claimed by the proposition.  Meanwhile, part (d) shows that there is $1$ such path for each endpoint $\eta_D$ corresponding to a disc $D \subset \cc_K$ minimally containing a cluster of odd cardinality $\neq 1, 2g + 1$.  Thus, exactly $o$ disjoint open segments are removed from the simply connected space $\Sigma_A'$ to get $\Sigma_{A, 0}$, and part (e) follows.

Part (d) says that each distinguished vertex of $\Sigma_{A, 0}$ corresponds to a disc $D \subset \cc_K$ which minimally contains a cluster of $A$ and which therefore satisfies $d(D) \in v(K^\times)$; it follows that the distance between any two distinguished vertices lies in $v(K^\times) \subset \qq$.  Moreover, part (d) tells us that distinct distinguished vertices in the same component $\Sigma' \subseteq \Sigma_{A, 0}$ must lie in distinct axes $\Lambda_{(i)}, \Lambda_{(j)}$.  If such distinct vertices had distance $\leq 2r$, then the midpoint of the path between them (which itself lies in $\Sigma' \subseteq \Sigma_{A, 0}$ by connectedness) would lie in both $B(\Lambda_{(i)}, r)$ and $B(\Lambda_{(j)}, r)$ for some indices $i \neq j$, contradicting $r$-separatedness.  Thus, part (f) is proved.

Finally, to prove part (g), let $\Sigma_0' \subseteq \Sigma_{A, 0}$ be a component as in the statement of part (e), and let $I \subseteq \{0, \dots, g\}$ be the subset of indices $i$ such that $\Sigma_0' \cap \Lambda_{(i)} \neq \varnothing$.  Let $\Sigma' = \bigsqcup_{i \in I} \Lambda_{(i)} \cup \Sigma_0'$.  As each axis $\Lambda_{(i)}$ as well as the space $\Sigma_0'$ is connected and $\Sigma'$ intersects each axis $\Lambda_{(i)}$ for $i \in I$, the space $\Sigma'$ is connected.  Therefore, to prove part (e) it clearly suffices to show that for any point $\eta \in \Sigma' \smallsetminus \bigsqcup_{i \in I} \Lambda_{(i)}$, the space $\Sigma' \smallsetminus \{\eta\}$ is not connected.  But this follows immediately from the fact that $\Sigma_A \smallsetminus \{\eta\}$ is not connected by construction of $\Sigma_A$, and the desired result follows.
\end{proof}

\subsection{Folding a reduced convex hull} \label{sec 3 folding}

In examining a set $S$ clustered in $\frac{v(p)}{p - 1}$-separated pairs through its reduced convex hull $\Sigma_{S, 0}$, we will focus on what happens when $S$ is altered by applying a fixed power of an order-$p$ generator of the corresponding $p$-Whittaker group to some subset of points in $S$.  In our setting, altering a set $S$ by applying a fixed power of an order-$p$ generator to a subset of points in $S$ has the effect of ``folding'' part of the reduced convex hull $\Sigma_{S, 0}$, which is reflected by the terminology introduced by the following definition.

\begin{dfn} \label{dfn folding}

Let $S = \{a_0, b_0, \dots, a_g, b_g\} \subset \proj_K^1$ be a subset which is clustered in $\frac{v(p)}{p - 1}$-separated pairs and define as above all of the associated objects given in \Cref{dfn associated to S} along with its reduced convex hull $\Sigma_{S, 0}$.  Choose distinct indices $i, j$ which satisfy the property that there is a distinguished vertex $v \in \Sigma_{S, 0} \cap \Lambda_{(j)}$ in the same component of $\Sigma_{S, 0}$ as $v_i$ (noting that the choice of $v$ given an index $j$ is unique thanks to \Cref{prop reduced convex hull}(d)).  Let $\tilde{v}$ be the unique point on the path $[v_i, v]$ which has distance $\frac{v(p)}{p - 1}$ from the end point $v$ (so that we have $\tilde{v} = v$ if the residue characteristic of $K$ is not $p$).  Let $I_{i, j} \subset \{0, \dots, g\}$ be the subset consisting of those indices $l$ such that the path $[\eta_{a_l}, \tilde{v}] \subset \Berk$ (or $[\eta_{b_l}, \tilde{v}] \subset \Berk$) intersects the path $[v_i, \tilde{v}]$ at an interior point.

A \emph{folding} of $S$ is a bijection of the form 
\begin{equation}
\phi = \phi_{(i, j, n)} : S \to S' := \{s_j^n(a_l), s_j^n(b_l)\}_{l \in I_{i, j}} \cup \{a_m, b_m\}_{m \notin I_{i, j}}
\end{equation}
 given by sending $a_l, b_l$ respectively to $s_j^n(a_l), s_j^n(b_l)$ for indices $l \in I_{i, j}$ and by fixing $a_m, b_m$ for indices $m \notin I_{i, j}$.
 
Consider the following two properties of a folding $\phi : S \to S'$.

\begin{enumerate}[(i)]

\item The subset $S' \subset \proj_K^1$ is clustered in $\frac{v(p)}{p - 1}$-separated pairs.

\item The intersection $[v_i, s_j^n(v_i)] \cap \Sigma_{S, 0}$ contains a sub-segment $[\tilde{v}_{-\epsilon}, \tilde{v}_\epsilon]$ with $\tilde{v}$ in its interior.

\end{enumerate}

We say that a folding $\phi : S \to S'$ is \emph{good} if it satisfies properties (i) and (ii) above.  We say that a folding $\phi : S \to S'$ is \emph{bad} if it fails to satisfy property (i) above.

We say that any $(2g + 2)$-element subsets $S, S' \subset \proj_K^1$ are \emph{folding equivalent} if there is a sequence of sets $S =: S_0, S_1, \dots, S_{m - 1}, S_m := S'$ such that for $1 \leq l \leq m$, there is a folding $\phi : S_{l - 1} \to S_l$ or $\phi : S_l \to S_{l - 1}$.

\end{dfn}

\begin{rmk} \label{rmk folding change of i}

It is immediate to see that in \Cref{dfn folding} above, if $\phi_{(i, j, n)} : S \to S'$ is a (resp. good, resp. bad) folding of $S$ for some indices $(i, j)$ and exponent $n$, then $\phi_{(i', j, n)} : S \to S'$ is also a (resp. good, resp. bad) folding of $S$ for any index $i' \in I_{i, j}$.

\end{rmk}

\begin{rmk} \label{rmk folding}

A folding of a set $S$ clustered in $\frac{v(p)}{p - 1}$-separated pairs is so called because of its effect on the reduced convex hull $\Sigma_{S, 0}$ of $S$.  Let $\tilde{v} \in \hat{\Lambda}_{(j)} \cap \Sigma_{S, 0}$ be the (unique) point on the path $[v, v']$ satisfying $\delta(\tilde{v}, v) = \frac{v(p)}{p - 1}$.  Let $T_{i, j} \subset \Sigma_{S, 0}$ be the set $(\tilde{T}_{i, j} \cap \Sigma_{S, 0}) \cup \{\tilde{v}\}$, where $\tilde{T}_{i, j}$ is the largest connected subspace of the (non-reduced) convex hull $\Sigma_S$ containing $v_i$ but not $\tilde{v}$.  An effect of producing a new set $S'$ and a folding $\phi = \phi_{i, j, n} : S \to S'$ is that $\Sigma_{S', 0}$ is obtained from $\Sigma_{S, 0}$ by taking the subspace $T_{i, j} \subsetneq \Sigma_{S, 0}$ and ``folding'' it ``across'' $\tilde{v}$ while leaving the rest of $\Sigma_{S, 0}$ unchanged.  More precisely, it induces a map $\phi_* = (\phi_{i, j, n})_* : \Sigma_{S, 0} \to \Sigma_{S'}$ (which is easily verified to be continuous) fixing $\Sigma_{S, 0} \smallsetminus T_{i, j}$ and acting as $s_j^n$ on $T_{i, j}$.  As long as the underlying set of $S'$ has even cardinality, one checks from \Cref{dfn reduced convex hull} that we have $\phi_*(\Sigma_A \smallsetminus \Sigma_{A, 1}) = \Sigma_{A'} \smallsetminus \Sigma_{A', 1}$; now it is an elementary exercise in point-set topology to check that the image of $\phi_*$ coincides with $\Sigma_{S', 0}$.  We can thus view $\phi_*$ as a map which ``changes'' the reduced convex hull of $S$ into that of $S'$.

If $\phi : S \to S'$ is a good folding, then part of the edge from $v_i$ to $v$ is moved under the induced ``folding operation'' $\phi_*$ to part of another edge on $\Sigma_{S, 0}$ coming out of $v$ that also passes through $\tilde{v}$.  Meanwhile, a bad folding $\phi : S \to S'$ is one such that a distinguished vertex of $\Sigma_{S, 0}$ is moved ``too close'' (\textit{i.e.} within a distance of $\frac{v(p)}{p-1}$) to another distinguished vertex of $\Sigma_{S, 0}$ under the map $\phi_*$.

The effect that a folding can have on a reduced convex hull $\Sigma_{S, 0}$ is depicted in \Cref{fig foldings} below.

\end{rmk}

\begin{figure}[h!]

\begin{subfigure}[b]{.3\textwidth}
\centering
\includegraphics[height=5.25cm]{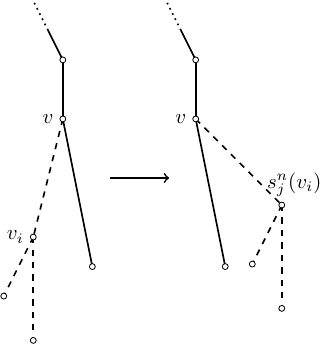}
\end{subfigure}
~
\begin{subfigure}[b]{.3\textwidth}
\centering
\includegraphics[height=5.25cm]{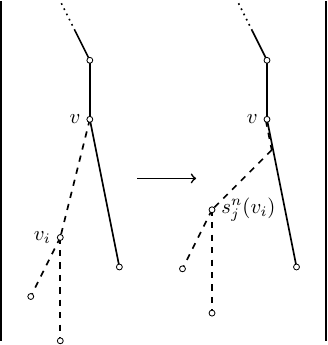}
\end{subfigure}
~
\begin{subfigure}[b]{.3\textwidth}
\centering
\includegraphics[height=5.25cm]{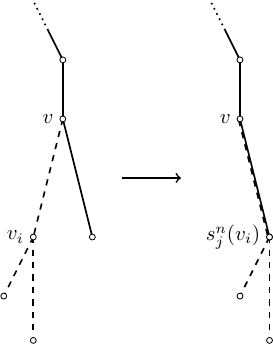}
\end{subfigure}

\caption{Each of these diagrams shows possible effects of a folding of a (different) set $S$ with respect to an ordered pair of indices $(i, j)$ (where in each one the distinguished vertex $v$ lies on the axis $\Lambda_{(j)}$) on part of the reduced convex hull $\Sigma_{S, 0}$.  In each case, the subspace of $\Sigma_{S, 0}$ shown in dashed lines is $T_{i, j}$ as defined in \Cref{rmk folding}.  The diagram on the left shows a folding which is not good or bad; the one in the middle shows a good folding; and the one on the right shows a bad folding.}

\label{fig foldings}
\end{figure}

The first crucial property of foldings that we show is that they do not affect the groups $\Gamma_0 \lhd \Gamma$ associated to a subset of $\proj_K^1$.

\begin{prop} \label{prop folding equivalence properties}

Assume the context of \Cref{dfn folding}, and suppose that $S' \subset \proj_K^1$ is a subset which is folding equivalent to $S$.  Writing $\Gamma \lhd \Gamma_0 < \PGL_2(K)$ and $\Gamma' \lhd \Gamma_0' < \PGL_2(K)$ for the subgroups respectively associated to $S$ and $S'$, we have $\Gamma_0' = \Gamma_0$ and $\Gamma' = \Gamma$.  Moreover, the set $S'$ is superelliptic if and only if $S$ is.

\end{prop}

\begin{proof}

Assume all of the notation used in \Cref{dfn folding}.  It is clear from the definition of folding equivalence in \Cref{dfn folding} that we may assume the existence of a folding $\phi_{i, j, n} : S \to S'$.  It is straightforward to see that for $l \in I_{i, j}$ (resp. $l \notin I_{i, j}$), we have $s_l' = s_j^n s_l s_j^{-n}$ (resp. $s_l' = s_l$).  From this it is clear that we have the inclusion $\Gamma_0' \subseteq \Gamma_0$.  At the same time, we get the reverse inclusion $\Gamma_0 \subseteq \Gamma_0'$ from the fact that $j \notin I_{i, j}$ and so we have $s_l = (s_j')^{-n} s_l' (s_j')^{n}$ for all $l \in I_{i, j}$.

Meanwhile, we claim that any element of $\Gamma_0 = \Gamma_0'$ which can be written as a product of powers of generators $s_l$ such that the sum of the exponents is $0$ lies in $\Gamma \lhd \Gamma_0$.  Indeed, this follows immediately from checking that for $1 \leq l \leq g$ we have $s_l^n s_0^{-n} = \gamma_{l, 1} \cdots \gamma_{l, n} \in \Gamma$ (with notation as in (\ref{eq Gamma})) and so every element of the form $s_l^n s_j^{-n}$ lies in $\Gamma$.  Now since each generator $(s_0')^{j - 1}s_i'(s_0')^{-j}$ of $\Gamma'$ satisfies this condition on exponents, we get the inclusion $\Gamma' \subseteq \Gamma$.  The reverse inclusion $\Gamma \subseteq \Gamma'$ follows from a similar symmetry argument as was used to show that $\Gamma_0 \subseteq \Gamma_0'$.  It is now immediate from the equalities $\Gamma_0' = \Gamma_0$ and $\Gamma' = \Gamma$ that $S'$ is superelliptic if and only if $S$ is.
\end{proof}

The next property of foldings that we show (in \Cref{prop lengths decrease} below) is that the points in their reduced convex hulls $\Sigma_{S', 0}$ are at least as close to each other as the points they correspond to in the original convex hull $\Sigma_{S, 0}$, and in the case that a folding $\phi: S \to S'$ satisfies property (ii) of \Cref{dfn folding}, some of the distances decrease, suggesting that the image of such a folding is in some sense an ``improvement'' on $S$.  As a corollary, we will get that bad foldings as well as good foldings of $S$ satisfy property (ii) (\Cref{cor bad foldings satisfy (ii)} below).  The property of being either a good folding or a bad folding is clearly an important one; we will often refer to a good folding or a bad folding of a set $S$ as a ``good or bad folding'' of $S$, and a number of hypotheses or conclusions in further results will involve a map $\phi: S \to S'$ either being a good or bad folding of $S$ or not being a good or bad folding of $S$.

\begin{prop} \label{prop lengths decrease}

Let $\phi = \phi_{i, j, n}: S \to S'$ be a folding and let $\phi_* : \Sigma_{S, 0} \to \Sigma_{S'}$ be the induced map defined as in \Cref{rmk folding}.

\begin{enumerate}[(a)]

\item For any vertices $w_1, w_2 \in \Sigma_{S, 0}$, we have 
\begin{equation} \label{eq lengths decrease inequality}
\delta(\phi_*(w_1), \phi_*(w_2)) \leq \delta(w_1, w_2).
\end{equation}

\item If the folding $\phi$ does not satisfy property (ii) in \Cref{dfn folding}, then equality holds in (\ref{eq lengths decrease inequality}) for all vertices $w_1, w_2 \in \Sigma_{S, 0}$.

\item If the folding $\phi$ does satisfy property (ii), then the inequality in (\ref{eq lengths decrease inequality}) is strict for precisely those pairs of vertices $w_1, w_2$ satisfying the following property: defining the subtree $T_{i, j} \subset \Sigma_{S, 0}$ as in \Cref{rmk folding} and defining the subtree $T_{i, j}' \subset \Sigma_{S', 0}$ to be the union of $\{\tilde{v}\}$ with the largest connected subspace $\tilde{T}_{i, j}' \subset \Sigma_{S', 0}$ which intersects the path $[s_j^n(v'), \tilde{v}]$ but does not contain $\tilde{v}$, we have $w_1 \in T_{i, j}$ and $w_2 \in T_{i, j}' \cap \Sigma_{S, 0}$.

\end{enumerate}

\end{prop}

\begin{proof}

We adopt the notation used in \Cref{dfn folding}, in particular for the points $v$ and $\tilde{v}$.  We claim first of all that if the folding $\phi$ does not satisfy property (ii) in \Cref{dfn folding}, then we have $T_{i, j}' \cap \Sigma_{S, 0} = \{\tilde{v}\}$.  Indeed, if there is a point $w \in \Sigma_{S, 0}$ such that $\tilde{v} \neq w \in T_{i, j}'$, then by construction of the space $T_{i, j}'$, we have that the paths $[\tilde{v}, s_j^n(v')]$ and $[\tilde{v}, w] \subset T_{i, j}' \cap \Sigma_{S, 0}$ intersect at a sub-segment of positive length, from which it follows that the intersection $[v', s_j^n(v')] \cap \Sigma_{S, 0}$ contains a sub-segment of positive length with $\tilde{v}$ as its interior, which is property (ii).

For any $w_1, w_2 \notin T_{i, j}$, the desired equality is immediate from the fact that $\phi_*$ fixes these points, and for any $w_1, w_2 \in T_{i, j}$, we get the same desired equality by noting that $\delta(\phi_*(w_1), \phi_*(w_2)) = \delta(s_j^n(w_1), s_j^n(w_2)) = \delta(v, w)$.  Meanwhile, for any vertex $w \notin T_{i, j}$ of $\Sigma_{S, 0}$, we have $\delta(\phi_*(v), \phi_*(w)) = \delta(s_j^n(v), w) = \delta(v, w)$ since $s_j$ fixes $v \in \hat{\Lambda}_{(j)}$.

Now assume without loss of generality that $w_1 \in T_{i, j}$ and $w_2 \notin T_{i, j}$.  It is clear from the definition of $T_{i, j}$ that any path from $w_1$ to $w_2$ must pass through the point $\tilde{v}$ given in \Cref{dfn folding}, and so we have 
\begin{equation} \label{eq distances decrease}
\delta(w_1, w_2) = \delta(w_1, \tilde{v}) + \delta(w_2, \tilde{v}).
\end{equation}
If $w_2 \notin T_{i, j}'$, then any path from $\phi_*(w_1) = s_j^n(w_1)$ to $\phi_*(w_2) = w_2$ similarly must pass through $\tilde{v}$, and so we get 
\begin{align}
\begin{split}
\delta(\phi_*(w_1), \phi_*(w_2)) &= \delta(s_j^n(w_1), w_2) = \delta(s_j^n(w_1), \tilde{v}) + \delta(w_2, \tilde{v}) \\
= \delta(s_j^n(w_1), s_j^n(\tilde{v})) &+ \delta(w_2, \tilde{v}) = \delta(w_1, \tilde{v}) + \delta(w_2, \tilde{v}) = \delta(w_1, w_2),
\end{split}
\end{align}
 using (\ref{eq distances decrease}) and the fact that $s_j$ fixes $\tilde{v} \in \hat{\Lambda}_{(j)}$.  Moreover, our above observation that $T_{i, j}' \cap \Sigma_{S, 0} = \{\tilde{v}\}$ implies that if the folding $\phi$ does not satisfy property (ii), we must have $w_2 \notin T_{i, j}'$, and so we have already proved part (a).
 
If on the other hand we have $w_2 \in T_{i, j}'$, then the intersection $\Sigma_{S, 0} \cap [w_1, w_2]$ must pass through $\tilde{v}$ and have nontrivial intersection with the path $[v_i, s_j^n(v_i)]$.  Then, after possibly replacing the endpoints $\tilde{v}_{-\epsilon}, \tilde{v}_\epsilon$ with those of a smaller segment containing $\tilde{v}$ in its interior, we get $[\tilde{v}_{-\epsilon}, \tilde{v}_\epsilon] \subset [v_i, s_j^n(v_i)]$ by property (ii) of \Cref{dfn folding}.  Then we have (again using (\ref{eq distances decrease})) 
\begin{align}
\begin{split}
\delta(\phi_*(w_1), \phi_*(w_2)) &= \delta(s_j^n(w_1), w_2) \leq \delta(s_j^n(w_1), \tilde{v}_\epsilon) + \delta(w_2, \tilde{v}_\epsilon) \\
 &< \delta(s_j^n(w_1), \tilde{v}) + \delta(w_2, \tilde{v}) = \delta(w_1, \tilde{v}) + \delta(w_2, \tilde{v}) = \delta(w_1, w_2).
\end{split}
\end{align}
This proves part (b).
\end{proof}

\begin{cor} \label{cor bad foldings satisfy (ii)}

If $\phi: S \to S'$ is a bad folding, then it satisfies property (ii) of \Cref{dfn folding}.  Equivalently, a folding $\phi: S \to S'$ is a good or bad folding if and only if it satisfies property (ii) of \Cref{dfn folding}.

\end{cor}

\begin{proof}

Suppose that $\phi = \phi_{i, j, n}: S \to S'$ is a folding that does not satisfy property (ii) of \Cref{dfn folding}; to prove the corollary, we need to show that $\phi$ is not a bad folding, or equivalently, that $S'$ is clustered in $\frac{v(p)}{p - 1}$-separated pairs.  By \Cref{prop lengths decrease}(a), we have that the map $\phi_* : \Sigma_{S, 0} \to \Sigma_{S', 0}$ defined in \Cref{rmk folding} preserves distances between points.  The convex hull $\Sigma_{S'}$ contains the axes $\Lambda_{a_l', b_l'}$ for $0 \leq i \leq g$, where each axis $\Lambda_{a_l', b_l'}$ is the image of the axis $\Lambda_{(l)}$ under either the identity map or the map $s_j^n$, so the distinguished vertices of $\Sigma_{S', 0}$ (lying in an axis $\Lambda_{a_l', b_l'}$) are images under $\phi_*$ of the distinguished vertices of $\Sigma_{S, 0}$ (lying in an axis $\Lambda_{(l)}$).  Using \Cref{prop reduced convex hull}(d) one can deduce the existence of a non-backtracking path with non-empty interior in $\Sigma_S$ between any pair of distinguished vertices of $\Sigma_{S', 0}$ lying in different axes, and so the analogous statement holds for $S'$.  By uniqueness of (non-backtracking) paths (see \Cref{rmk paths in berk}), we then get that the axes $\Lambda_{a_l', b_l'}$ are mutually disjoint, so the set $S'$ is clustered in pairs.  By \Cref{prop reduced convex hull}(f), the distance between any distinct distinguished vertices in the same component of $\Sigma_{S, 0}$ is $> 2\frac{v(p)}{p - 1}$, and therefore the same is true for $\Sigma_{S', 0}$.  It is then easy to deduce using \Cref{prop reduced convex hull}(a), and the fact that distinguished vertices in the same axis must lie in different components of $\Sigma_{S', 0}$ by \Cref{prop reduced convex hull}(d), that we must have $\delta(\Lambda_{a_l', b_l'}, \Lambda_{a_m', b_m'}) > 2\frac{v(p)}{p - 1}$ for distinct indices $l, m$.  The desired conclusion that $S'$ is clustered in $\frac{v(p)}{p - 1}$-separated pairs directly follows.
\end{proof}

In the following proposition and its proof, we will refer to a partial order on $\Berk$ defined as in \cite[\S1.4.1]{baker2008introduction}, where in particular, points in $\Berk$ of Type II and III are ordered according to inclusion of their corresponding subsets of $\cc_K$, and given distinct points $\eta, \eta' \in \Berk \smallsetminus \{\eta_\infty\}$, we have $\eta > \eta'$ if and only if $\eta$ lies in the interior of the path $[\eta', \eta_{\infty}]$.  We observe that all points greater than a fixed point $\eta \in \Berk$ lie in the same well-ordered path $[\eta, \eta_\infty] \subset \Berk$ and freely use this fact below.

\begin{prop} \label{prop good or bad folding}

Let $S \subset \proj_K^1$ be a subset which is clustered in $\frac{v(p)}{p - 1}$-separated pairs with all of its associated objects.  Let $\eta \in \Berk$ be any point, and suppose that for some index $j$ and some $n \in \zz \smallsetminus p\zz$, the path $[\eta, s_j^n(\eta)] \subset \Berk$ contains a sub-segment $[\tilde{v}_{-\epsilon}, \tilde{v}_\epsilon] \subset \Sigma_{S, 0}$ with $\tilde{v}$ in its interior, where $\tilde{v}$ is the closest point in $\hat{\Lambda}_{(j)}$ to $\eta$.  Then there is a good or bad folding $\phi = \phi_{i, j, m} : S \to S'$ for some index $i$ and with $m \in \{\pm n\}$.  We may moreover choose $i$ so that we have $v_i < \tilde{v}$.

\end{prop}

\begin{proof}

Assume the hypotheses of the statement, and let $\xi \in \Sigma_{S, 0}$ be the closest point in $\Sigma_{S, 0}$ to $\eta$.  Let $v \in \Sigma_{S, 0}$ be the (unique thanks to \Cref{prop reduced convex hull}(d)) distinguished vertex on the same component as $\xi$ which lies in the axis $\Lambda_{(j)}$.

Assume for now that we have $\xi < \tilde{v}$ under the partial order discussed above.  Then it follows from the real tree structure of $\Sigma_{S, 0}$ and from \Cref{prop reduced convex hull}(c) that $\xi$ is a distinguished vertex lying in the same connected component of $\Sigma_{S, 0}$ as $\tilde{v}$.  Let $i$ be the index such that we have $\xi \in \Lambda_{(i)}$.  Let $D' \subset \cc_K$ be the disc such that $v' = \eta_{D'}$.  By our above observations on the parial order, any sufficiently small disc $D'' \supsetneq D'$ satisfies that $\eta_{D''}$ is an interior point of $[\xi, v]$.  If the intersection $D' \cap A$ were an odd-cardinality cluster of $A$, then the same could be assumed true of $D'' \cap A$ and then we would have $\eta_{D''} \notin \Sigma_{S, 0}$; thus, we have that $\#(D' \cap A)$ is even and therefore $\xi = v_i$.

Now we show that the folding $\phi_{i, j, n}$ of $S$ is a good or bad folding.  It follows from \Cref{prop loxodromic action on berk}(d) that we have an inclusion of paths $[v_i, s_j^n(v_i)] \subseteq [\eta, s_j^n(\eta)]$.  Then the hypothesis that $[\tilde{v}_{-\epsilon}, \tilde{v}_\epsilon] \subseteq [\eta, s_j^n(\eta)]$ implies that, after possibly replacing the end points $\tilde{v}_{-\epsilon}, \tilde{v}_\epsilon$ with those of a smaller segment containing $\tilde{v}$ in its interior, we have $[\tilde{v}_{-\epsilon}, \tilde{v}_\epsilon] \subseteq [v_i, s_j^n(v_i)]$.   Thus, the folding in question satisfies property (ii) of \Cref{dfn folding}, and by \Cref{cor bad foldings satisfy (ii)}, it is a good or bad folding.  The proposition is thus proved in the case that $\xi < \tilde{v}$.

Now assume that we do not have $\xi < \tilde{v}$.  Then by our above observations on the partial order, after possibly replacing the end points $\tilde{v}_{-\epsilon}, \tilde{v}_\epsilon$ with those of a smaller segment containing $\tilde{v}$ in its interior, we have $\tilde{v}_{-\epsilon} > \tilde{v}$, and this means that we must have $\tilde{v}_\epsilon < \tilde{v}$ since the segment $[\tilde{v}_{-\epsilon}, \tilde{v}_\epsilon]$ is a non-backtracking path.  We therefore have $s_j^n(\xi) < \tilde{v}$, and essentially the same argument of the previous two paragraphs can be used to show that there is an index $i$ such that $v_i = s_j^n(\xi) < \tilde{v}$ and such that the folding $\phi_{i, j, -n}$ of $S$ is a good or bad folding.  This completes the proof of the proposition.
\end{proof}

\subsection{Optimal sets and finding them using good foldings} \label{sec 3 optimal}

Our strategy in trying to determine whether or not a set $S$ is superelliptic will hinge on attempting to ``fold'' $S$ into a set that satisfies the following condition.

\begin{dfn} \label{dfn optimal}

A $(2g + 2)$-element subset $S \subset \proj_K^1$ clustered in $\frac{v(p)}{p - 1}$-separated pairs is \emph{optimal} if there does not exist a good or bad folding of $S$.

\end{dfn}

In our below statements about an optimal set $S$, we implicitly assume the labelings $\{a_i, b_i\}$ for the pairs that $S$ is clustered in (the choice of which is uniquely determined by \Cref{rmk clustered in pairs}(b)), along with the associated notation for axes and distinguished vertices.

We are ready to state our main result of this paper, which characterizes a superelliptic set in terms of foldings of its reduced convex hull.

\begin{thm} \label{thm criterion for being superelliptic}
 
Let $S \subset \proj_K^1$ be a $(2g + 2)$-element subset clustered in $\frac{v(p)}{p - 1}$-separated pairs.  Then $S$ is $p$-superelliptic if and only if there exists an optimal set $\Smin$ which is folding equivalent to $S$.

Moreover, when these equivalent properties hold, the optimal set $\Smin$ also has the property that no distinct points of its reduced convex hull $\Sigma_{\Smin, 0}$ lie in the same orbit under the action of the Schottky group $\Gamma$ associated to $S$.

\end{thm}

In order to prove this theorem, we first need some lemmas.

\begin{lemma} \label{lemma algorithm terminates}

For any $(2g + 2)$-element subset $S \subset \proj_K^1$ clustered in $\frac{v(p)}{p - 1}$-separated pairs, there is no infinite sequence $S =: S_0, S_1, S_2, \dots$ of such subsets such that there is a good folding $\phi^{(n)}: S_{n - 1} \to S_n$ for all $n \geq 1$.  In other words, every sequence of such subsets where each is a good folding of the last must terminate.

\end{lemma}

\begin{proof}

Let $\{S_n\}_{n \geq 0}$ be such an infinite sequence, with good foldings $\phi^{(n)}$ for $n \geq 1$ as in the statement.  For each $n$, let $D_n$ be the $\frac{g(g + 1)}{2}$-element multiset of distances in $S_n$ between each pair of distinct distinguished vertices of $\Sigma_{S_n, 0}$.  \Cref{prop lengths decrease}(a) implies that for each $n \geq 0$, there is a bijection from the multiset $D_n$ to the multiset $D_{n+1}$ sending each element of $D_n$ to an equal or lesser positive rational number.  Meanwhile, since $\phi^{(n)}$ is a good folding, \Cref{prop lengths decrease}(c) implies that some distance decreases (as the trees $T_{i, j}, T_{i, j}'$ defined in that proposition by construction each contain some distinguished vertex of $\Sigma_{S, 0}$) and so we even have $D_n \neq D_{n + 1}$.  But these elements all lie in the fixed discrete subgroup $v(K^\times) \subset \qq$ by \Cref{prop reduced convex hull}(f), and this immediately implies the desired contradiction.
\end{proof}

\begin{lemma} \label{lemma off the convex hull}

Suppose that $S \subset \proj_K^1$ is an optimal subset.  Let $\eta \in \Berk$ be any point of Type II or III, and let $\xi$ be the closest point in the reduced convex hull $\Sigma_{S, 0}$ to $\eta$.  Let $j \in \{0, \dots, g\}$ be an index such that $\xi \notin \hat{\Lambda}_{(j)}$, and let $\tilde{v}$ be the closest point in $\hat{\Lambda}_{(j)}$ to $\eta$.
Then for any exponent $n \in \zz \smallsetminus p\zz$, we have 
\begin{equation}
\delta(s_j^n(\eta), \Sigma_S) = \delta(s_j^n(\eta), \hat{\Lambda}_{(j)}) = \delta(\eta, \xi) + \delta(\xi, \hat{\Lambda}_{(j)}) > \delta(\eta, \Sigma_S),
\end{equation}
 and the closest point in $\Sigma_S$ to $s_j^n(\eta)$ is $\tilde{v} \in \hat{\Lambda}_{(j)}$.

\end{lemma}

\begin{proof}

Since $\xi \notin \hat{\Lambda}_{(j)}$, \Cref{prop reduced convex hull}(b) implies that we also have $\eta \notin \hat{\Lambda}_{(j)}$.  Then we must have $\tilde{v} \in \Sigma_{S, 0}$, because otherwise $[\eta, \tilde{v}] \cup [\tilde{v}, \xi]$ would be a path without backtracking; this would imply that $\xi$ is the closest point in $\Sigma_{S, 0}$ to $\tilde{v}$ and, by \Cref{prop reduced convex hull}(b), that would contradict the fact that $\xi \notin \hat{\Lambda}_{(j)}$.

Let $w$ be the closest point in $\Sigma_S$ to $s_j^n(\eta)$.  \Cref{prop loxodromic action on berk}(d) implies that the path $[\eta, s_j^n(\eta)] \subset \Berk$ contains $\tilde{v}$ in its interior and that $\tilde{v}$ is the closest point in $\hat{\Lambda}_{(j)}$ to $s_j^n(\eta)$; the points $\eta, \xi, \tilde{v}, w, s_j^n(\eta) \subset [\eta, s_j^n(\eta)]$ all then share the same closest point in the axis $\Lambda_{(j)}$.  Now \Cref{prop reduced convex hull}(d) implies that the points $\xi, \tilde{v}, w$ all lie in the same component of $\Sigma_{S, 0}$.  Suppose that $w \neq \tilde{v}$.  Then since $w$ is in the interior of the path from $s_j^n(\eta)$ to its closest point $\tilde{v}$ in $\hat{\Lambda}_{(j)}$, we have $w \notin \hat{\Lambda}_{(j)}$; \Cref{prop reduced convex hull}(a) then implies the existence of a point $w' \in [\tilde{v}, w]$ with $w' \in \Sigma_{S, 0}$.  Now $\tilde{v}$ lies in the interior of the path $[\xi, w'] \subset \Sigma_{S, 0}$, and since $S$ is assumed to be optimal, the inclusion of paths $[\tilde{v}_{-\epsilon} := \xi, \tilde{v}_\epsilon := w'] \subset [\eta, s_j^n(\eta)]$ is a contradiction due to \Cref{prop good or bad folding}.  Therefore, the point $\tilde{v} = w$ is itself the closest point in $\Sigma_S$ to $s_j^n(\eta)$.  We therefore have (using \Cref{prop order-p elliptic}(b)) 
\begin{equation}
\begin{aligned}
\delta(s_j^n(\eta), \Sigma_{S, 0}) = \delta(s_j^n(\eta), \tilde{v}) = \delta&(s_j^n(\eta), s_j^n(\tilde{v})) = \delta(\eta, \tilde{v}) = \delta(\eta, \xi) + \delta(\xi, \tilde{v}) \\
&= \delta(\eta, \xi) + \delta(\xi, \hat{\Lambda}_{(j)}) \geq \delta(\eta, \Sigma_S) + \delta(\xi, \hat{\Lambda}_{(j)}) > \delta(\eta, \Sigma_S).
\end{aligned}
\end{equation}
\end{proof}

\begin{lemma} \label{lemma off the convex hull2}

Suppose that $S \subset \proj_K^1$ is an optimal subset.  Choose any $\eta \in \Berk$ of Type II or III and any nontrivial element $\gamma \in \Gamma_0$, which we write as a product 
\begin{equation} \label{eq word}
\gamma = s_{i_t}^{n_t} s_{i_{t - 1}}^{n_{t - 1}} \cdots s_{i_1}^{n_1}
\end{equation}
 for some $t \geq 1$, some $n_1, \dots, n_t \in \zz \smallsetminus p\zz$, and some indices $i_l$ satisfying $i_l \neq i_{l - 1}$ for $2 \leq l \leq t$.
 
Suppose that the closest point in $\Sigma_S$ to $\eta$ does not lie in $\hat{\Lambda}_{(i_1)}$.  Then the closest point in $\Sigma_S$ to $\gamma(\eta)$ lies in $\hat{\Lambda}_{(i_t)}$, and in fact we have 
\begin{equation} \label{eq distance from gamma(eta) to hull}
\delta(\gamma(\eta), \Sigma_S) \geq \delta(\eta, \hat{\Lambda}_{(i_1)}) + \sum_{l = 2}^t \delta(\hat{\Lambda}_{(i_{l-1})}, \hat{\Lambda}_{(i_l)}) > \delta(\eta, \Sigma_S).
\end{equation}

\end{lemma}

\begin{proof}

We first observe that our hypothesis about the closest point in $\Sigma_S$ to $\eta$ implies that the closest point in $\Sigma_{S, 0}$ to $\eta$ also is not in $\hat{\Lambda}_{(i_1)}$ by \Cref{prop reduced convex hull}(b).  Moreover, our hypothesis that $S$ is clustered in $\frac{v(p)}{p-1}$-separated pairs says that we have $\hat{\Lambda}_{(i_l)} \cap \hat{\Lambda}_{(i_{l-1})} = \varnothing$ for $2 \leq l \leq t$.  It follows that for any such index $l$, if the closest point $\xi'$ in $\Sigma_{S, 0}$ to some point $\eta' \in \Berk$ lies in $\hat{\Lambda}_{(i_{l-1})}$, then we have $\xi' \notin \hat{\Lambda}_{(i_l)}$.  Using these observations, we may apply \Cref{lemma off the convex hull} to $\eta$ and inductively to intermediate points $\eta_l := s_{i_l}^{n_l} \cdots s_{i_1}^{n_1}(\eta)$ to get the formula in (\ref{eq distance from gamma(eta) to hull}) (the inequality there comes from the fact that the the closest point in $\Sigma_{S, 0}$ to $\eta$ does not lie in $\hat{\Lambda}_{(i_1)}$ and so we have $\delta(\eta, \hat{\Lambda}_{(i_1)}) > \delta(\eta, \Sigma_{S, 0})$), that we have $\gamma(\eta) \notin \hat{\Lambda}_{(i_t)}$, and that the closest point in $\Sigma_{S, 0}$ to $\gamma(\eta)$ lies in $\hat{\Lambda}_{(i_t)}$.
\end{proof}

\begin{cor} \label{cor optimal free}

Suppose that $S \subset \proj_K^1$ is an optimal subset.  Then there are no group relations among the associated order-$p$ elements $s_0, \dots, s_g$ apart from $s_0^p = \cdots = s_g^p = 1$.

\end{cor}

\begin{proof}

Let $\gamma \subset \Gamma_0 = \langle s_0, \dots, s_g \rangle$ be any element written as a word as in (\ref{eq word}) in the statement of \Cref{lemma off the convex hull2} above for some $t \geq 1$; it suffices to show that $\gamma \neq 1$.  But this follows immediately from \Cref{lemma off the convex hull2}, as we can always choose a point $\eta \in \Berk$ whose closest point in the reduced convex hull $\Sigma_{S, 0}$ is not in $\hat{\Lambda}_{(i_1)}$, so that the conclusion of that lemma implies that $\gamma(\eta) \neq \eta$.
\end{proof}

\begin{lemma} \label{lemma optimal is superelliptic}

An optimal subset $S \subset \proj_K^1$ is superelliptic, and it has the property that no distinct points of its reduced convex hull $\Sigma_{S, 0}$ lie in the same orbit under the action of the associated Schottky group $\Gamma$.

\end{lemma}

\begin{proof}

Assume that $S$ is optimal, and let $\Gamma \lhd \Gamma_0 < \PGL_2(K)$ be the subgroups associated to $S$.  It is immediate from \Cref{cor optimal free} that no proper subset of $\{s_0, \dots, s_g\}$ generates $\Gamma_0 = \langle s_0, \dots, s_g \rangle$; to see that $S$ is superelliptic, we now only have to prove that every nontrivial element of $\Gamma$ is loxodromic.  We will show that given any point $\eta \in \Berk$ of Type II or III and any nontrivial element $\gamma \in \Gamma$, we have $\gamma(\eta) \neq \eta$, which by \Cref{cor loxodromic doesn't fix points} now implies the first statement of the lemma.  We will show moreover that for such a point $\eta$ and element $\gamma$, if $\eta \in \Sigma_S$, we have $\gamma(\eta) \notin \Sigma_S$, which is equivalent to the second statement of the lemma.

Choose any $\eta \in \Berk$ of Type II or III and any nontrivial element $\gamma \in \Gamma$, which we write as a word as in (\ref{eq word}).  We do not have $t = 1$, since then this element would lie in $\Gamma_0 \smallsetminus \Gamma$.  Moreover, after possibly replacing $\gamma$ with a conjugate (in $\Gamma_0 \rhd \Gamma$) of $\gamma$ (which does not affect the property of being loxodromic), we may assume that $i_t \neq i_1$.

First suppose that the closest point in $\Sigma_S$ to $\eta$ does not lie in $\hat{\Lambda}_{(i_1)}$.  Then \Cref{lemma off the convex hull2} implies that we have $\gamma(\eta) \neq \eta$ and that $\eta \in \Sigma_S$ implies $\gamma(\eta) \notin \Sigma_S$, and we are done.

In the case that $\eta \in \Sigma_S \cap \hat{\Lambda}_{(i_1)}$, we have $\gamma(\eta) = \gamma s_{i_1}^{-n_1}(\eta)$ and we may replace $\gamma$ with $\gamma s_{i_1}^{-n_1} = s_{i_t}^{n_t} \cdots s_{i_2}^{n_2}$; now the closest point in $\Sigma_S$ to $\eta$ does not lie in $\hat{\Lambda}_{i_2}$ since $\hat{\Lambda}_{(i_1)} \cap \hat{\Lambda}_{i_2} = \varnothing$ from the fact that $S$ is clustered in $\frac{v(p)}{p - 1}$-separated pairs, and we may use the argument in the previous paragraph to show that we again have $\gamma(\eta) \notin \Sigma_S$.  This completes the proof of the second statement of the lemma.

Now, letting $\gamma_{(r)} = s_{i_r}^{n_r} \cdots s_{i_1}^{n_1}$ for $1 \leq r \leq t$, suppose more generally that there is some $r$ such that the closest point in $\Sigma_S$ to $\gamma_{(r)}(\eta)$ does not lie in $\hat{\Lambda}_{r + 1}$.  Then, by the same use of \Cref{lemma off the convex hull2} applied to the element $\gamma_{(r)} \gamma \gamma_{(r)}^{-1} = s_{i_r}^{n_r} \cdots s_{i_1}^{n_1} s_{i_t}^{r_t} \cdots s_{i_{r+1}}^{n_{r+1}} \in \Gamma$ (which does not reduce to a shorter word thanks to our assumption that $i_t \neq i_1$) and point $\gamma_{(r)}(\eta)$, we get 
\begin{equation} \label{eq rotated word}
\gamma_{(r)} \gamma \gamma_{(r)}^{-1}(\gamma_{(r)}\eta) \neq \gamma_{(r)}(\eta).
\end{equation}
Multiplying both sides above by $\gamma_{(r)}^{-1}$ and simplifying yields $\gamma(\eta) \neq \eta$, and again we are done.

Finally, suppose that we are not in either of the above two cases.  Then in particular the closest point in $\Sigma_S$ to $\gamma_{(t - 1)}(\eta)$ lies in $\hat{\Lambda}_{(i_t)}$.  Then the closest point in $\Sigma_S$ to $\gamma_{(t - 1)}(\eta)$ does not lie in $\hat{\Lambda}_{(i_{t-1})}$, as we have $\hat{\Lambda}_{(i_t)} \cap \hat{\Lambda}_{(i_{t-1})} = \varnothing$ from the fact that $S$ is clustered in $\frac{v(p)}{p - 1}$-separated pairs.  Then, by the same use of \Cref{lemma off the convex hull2} applied to the element $s_{i_t}^{-n_t}\gamma^{-1}s_{i_t} = s_{i_t}^{-n_t} s_{i_1}^{-n_1} \cdots s_{i_{t-1}}^{n_{t-1}} \in \Gamma$ (which does not reduce to a shorter word thanks to our assumption that $i_t \neq i_1$) and the point $\gamma_{(t - 1)}(\eta)$, we get 
\begin{equation} \label{eq rotated inverted word}
s_{i_t}^{-n_t}\gamma^{-1}s_{i_t}(\gamma_{(t - 1)}(\eta)) \neq \gamma_{(t - 1)}(\eta).
\end{equation}
Multiplying both sides above by $s_{i_t}^{n_t}$ and simplifying yields $\eta \neq \gamma(\eta)$, and again we are done.
\end{proof}

\begin{lemma} \label{lemma bad foldings are bad}

Let $S \subset \proj_K^1$ be a subset which is clustered in $\frac{v(p)}{p - 1}$-separated pairs.  If there exists a bad folding of $S$, then $S$ is not superelliptic.

\end{lemma}

\begin{proof}

Letting $s_0, \dots, s_g \in \PGL_2(K)$ be the order-$p$ elements associated to $S$, the existence of a bad folding of $S$ implies that for some indices $i, j, l \in \{0, \dots, g\}$ with $i \neq l$ and some $n \in \zz \smallsetminus p\zz$, we have $s_j^n(\hat{\Lambda}_{(i)}) \cap \hat{\Lambda}_{(l)} \neq \varnothing$.  Then according to \Cref{prop clustered in pairs}, the set $S$ is not superelliptic.
\end{proof}

We are finally ready to prove \Cref{thm criterion for being superelliptic}.

\begin{proof}[Proof (of \Cref{thm criterion for being superelliptic})]

Assume first that there is an optimal set $\Smin$ which is folding equivalent to $S$.  \Cref{lemma optimal is superelliptic} implies, by property (i) in \Cref{dfn superelliptic set}, that the subgroups $\Gamma \lhd \Gamma_0 < \PGL_2(K)$ induced by $\Smin$ are respectively Schottky and $p$-Whittaker.  According to \Cref{prop folding equivalence properties}, these are the same subgroups of $\PGL_2(K)$ as the ones associated to $S$, and so $S$ satisfies property (i) of \Cref{dfn superelliptic set}.  Now \Cref{cor optimal free} implies that $\Gamma_0$ is isomorphic to the free product of $g + 1$ copies of $\zz / p\zz$.  It is well known that such a group cannot be generated by fewer than $g + 1$ elements (see \cite{neumann1943number}), and so property (ii) of \Cref{dfn superelliptic set} is also satisfied for $S$; we conclude that $S$ is superelliptic.

Now assume conversely that $S$ is superelliptic.  We define a sequence of subsets $S_0, S_1, S_2, \dots \subset \proj_K^1$ recursively as follows.  Set $S_0 = S$, and for each $n \geq 0$, if there is a good folding of $S_n$, choose one and let $S_{n+1}$ be the image of $S_n$ under it.  Now \Cref{lemma algorithm terminates} implies that there is some $N \geq 0$ such that a set $S_N$ is defined but there is no good folding of $S_N$.  By \Cref{prop folding equivalence properties}, the set $S_N$ is also superelliptic.  Now applying \Cref{lemma bad foldings are bad} to $S_N$ shows that there does not exist a bad folding of $S_N$ either.  Therefore $\Smin := S_N$ is optimal, and we have finished proving of the first statement of the theorem.

The second statement of the theorem is simply a restatement of that of \Cref{lemma optimal is superelliptic} applied to $\Smin$ (noting again that $S$ and $\Smin$ determine the same Schottky group $\Gamma$ by \Cref{prop folding equivalence properties}(b)).
\end{proof}

In the course of proving \Cref{thm criterion for being superelliptic}, we have pointed the way toward proving some other useful facts, which we note in the remarks below.

\begin{rmk} \label{rmk free}

From our above arguments, none of which assumed any knowledge of the group structures of $\Gamma \lhd \Gamma_0$ \textit{a priori}, we are able in the superelliptic case to recover the well-known fact that Schottky groups are free.  Indeed, a superelliptic $(2g + 2)$-element set $S$ is folding equivalent to an optimal set $\Smin$ by \Cref{thm criterion for being superelliptic}, and we see from \Cref{cor optimal free} that the $p$-Whittaker group $\Gamma_0$ associated to $\Smin$ (which is also associated to $S$ by \Cref{prop folding equivalence properties}(b)) is isomorphic to the free product of $g + 1$ order-$p$ subgroups.  It is easy to deduce from this that there are no relations among the order-$p$ elements $s_i \in \PGL_2(K)$ which are associated to $S$ apart from $s_0^p = \cdots = s_g^p = 1$.  Then it is elementary to show (as is done in \cite[\S2]{van1982galois}) that the elements $\gamma_{i, j} \in \Gamma$ given in (\ref{eq Gamma}) freely generate $\Gamma$.  Thus, the Schottky group $\Gamma$ is a free group on $(p - 1)g$ generators, and so by \cite[Theorem III.2.2]{gerritzen2006schottky}, the resulting superelliptic curve $C \cong \Omega / \Gamma$ has genus equal to $(p - 1)g$.

\end{rmk}

\begin{rmk} \label{rmk converse}

Using the group structure of the $p$-Whittaker group $\Gamma_0$ associated to a superelliptic subset $S \subset \proj_K^1$ given by \Cref{rmk free}, it is easy to appropriately alter the argumentation in the proof of \Cref{prop clustered in pairs} to prove a variant (easily shown to be stronger) of the statement as follows: with the same set-up as in the statement, suppose that we have $\gamma(\hat{\Lambda}_{(i)}) \cap \hat{\Lambda}_{(j)} \neq \varnothing$ for some $\gamma \in \Gamma$ and $i, j \in \{0, \dots, g\}$.  Then we have $\gamma = 1$ and $i = j$.

The arguments in the proof of \Cref{thm criterion for being superelliptic} show that a converse to the above is also true: if $S \subset \proj_K^1$ is a subset such that for all elements $\gamma \in \Gamma$ in the associated Schottky group and for all pairs of distinct indices $i, j \in \{0, \dots, g\}$, we have $\gamma(\hat{\Lambda}_{(i)}) \cap \hat{\Lambda}_{(j)} = \varnothing$, then $S$ is superelliptic.  Indeed, given a set $S$ satisfying this disjointness hypothesis, we may construct a sequence of good foldings 
\begin{equation}
S =: S_0 \longrightarrow S_1 \longrightarrow \dots \longrightarrow S_N
\end{equation}
 as in that proof, such that there is no good folding of $S_N$.  Then it is not hard to see that the disjointness hypothesis implies that there is no bad folding of $S_N$ either, so that $\Smin := S_N$ is optimal.  Then \Cref{thm criterion for being superelliptic} says that $S$ is superelliptic.

\end{rmk}

\subsection{Special cases of superelliptic subsets of $\proj_K^1$} \label{sec 3 special cases}

As somewhat of a corollary to the methods and results we have obtained in the above subsection, we are now able to isolate several special situations in which it is easy to determine whether or not a subset $S \subset \proj_K^1$ is superelliptic, and in particular, we are able to recover a result of Kadziela.

\begin{prop} \label{prop valency}

Let $S \subset \proj_K^1$ be a subset which is clustered in $\frac{v(p)}{p - 1}$-separated pairs.  Suppose that there is a good or bad folding $\phi_{i, j, n}$ of $S$, and let $v \in \Lambda_{(j)}$ be the (unique thanks to \Cref{prop reduced convex hull}(d)) distinguished vertex lying in the same component of $\Sigma_{S, 0}$ as $v_i$.  Then if the residue characteristic of $K$ is not $p$, the distinguished vertex $v$ has valency $\geq 2$; if the residue characteristic of $K$ is $p$, there is a vertex $\tilde{v}$ which has valency $\geq 3$ and satisfies $\delta(v, \tilde{v}) = \frac{v(p)}{p - 1}$.

\end{prop}

\begin{proof}

Assume the notation in \Cref{dfn folding}, in particular that $j$ is the index such that $v \in \Lambda_{(j)}$ and that $\tilde{v} \in \Sigma_{S, 0}$ is the closest point in $\hat{\Lambda}_{(j)}$ to $v'$.  A good or bad folding $\phi_{i, j, n}$ of $S$ satisfies condition (ii) in \Cref{dfn folding} by \Cref{cor bad foldings satisfy (ii)}; this implies in particular that there is a segment $[\tilde{v}_{-\epsilon}, \tilde{v}_\epsilon]$ of a non-backtracking path in $\Sigma_{S, 0}$ which contains the point $\tilde{v}$ in its interior.  If the residue characteristic of $K$ is not $p$, then we have $\tilde{v} = v$ and the path $[\tilde{v}_{-\epsilon}, \tilde{v}_\epsilon] \subset \Sigma_{S, 0}$ contains sub-segments of $2$ edges coming out of $v$, hence the desired statement.  If the residue characteristic of $K$ is not $p$, then since it follows from \Cref{prop loxodromic action on berk}(d) that $[\tilde{v}_{-\epsilon}, \tilde{v}_\epsilon] \cap \hat{\Lambda}_{(j)} = \{\tilde{v}\}$, this segment contains sub-segments of $2$ edges coming out of $\tilde{v}$ which are each distinct from the path connecting $\tilde{v}$ to the distinguished vertex $v$; the desired statement again follows.
\end{proof}

\begin{cor} \label{cor valency}

If the reduced convex hull $\Sigma_{S, 0}$ of a subset $S \subset \proj_K^1$ which is clustered in $\frac{v(p)}{p - 1}$-separated pairs satisfies the property that each of its distinguished vertices is a tail (\textit{i.e.} has valency $1$), then $S$ is optimal (and thus also superelliptic by \Cref{lemma optimal is superelliptic}).

\end{cor}

\begin{proof}

This is immediate from \Cref{prop valency}.
\end{proof}

The following proposition allows us to simplify the problem of determining whether an even-cardinality subset $S \subset \proj_K^1$ is superelliptic or optimal by considering subsets of $S$ corresponding to connected components of its reduced convex hull.

\begin{prop} \label{prop components}

Let $S \subset \proj_K^1$ be a subset clustered in $\frac{v(p)}{p - 1}$-separated pairs, and write the decomposition of its complex hull into connected components as $\Sigma_{S, 0} = \Sigma_{S, 0}^{(1)} \sqcup \dots \sqcup \Sigma_{S, 0}^{(M)}$.  For $1 \leq m \leq M$, let $I^{(m)} \subset \{0, \dots, g\}$ be the subset consisting of the indices $i$ such that we have $\Sigma_{S, 0}^{(m)} \cap \Lambda_{(i)} \neq \varnothing$ and write $S_m = \{a_i, b_i \ | \ i \in I^{(m)}\} \subset S$.  Then the set $S$ is superelliptic (resp. optimal) if and only if $S_m$ is superelliptic (resp. optimal) for $1 \leq m \leq M$.

\end{prop}

\begin{proof}

\Cref{prop reduced convex hull}(g) says that for $1 \leq m \leq M$, the reduced convex hull $\Sigma_{S_m, 0}$ of $S_m$ can be identified with the component $\Sigma_{S, 0}^{(m)} \subset \Sigma_{S, 0}$.  Now it is clear from \Cref{prop valency} that there is a good or bad folding $\phi_{i, j, n}$ of $S$ if and only if there is a good or bad folding $\phi_{i, j, n}$ of $S_m$, where $m$ is the index such that $v_i \in \Sigma_{S, 0}^{(m)} = \Sigma_{S_m, 0}$.  The claim for the property of being optimal follows.

Let $(i, j)$ be an ordered pair of vertices such that the component $\Sigma_{S, 0}^{(m)}$ which contains $v_i$ satisfies $\Sigma_{S, 0}^{(m)} \cap \Lambda_{(j)} \neq \varnothing$, and define the set $I_{i, j}$ as in \Cref{dfn folding}.  Given an index $m' \neq m$, we have $S_{m'} \neq \{j\}$, because otherwise by applying \Cref{prop reduced convex hull}(c)(d) we would get that $\Sigma_{S, 0}^{(m)}$ consists of an isolated distinguished point, and by construction there are no isolated points in $\Sigma_{S, 0}$.  Now it follows from definitions that we have either $S_{m'} \subset I_{i, j} \cup \{j\}$ or $S_{m'} \cap I_{i, j} \subseteq \{j\}$.  Given a folding $\phi_{i, j, n} : S \to S'$, note that we may write the decomposition $\Sigma_{S', 0} = \Sigma_{S', 0}^{(1)} \sqcup \dots \sqcup \Sigma_{S', 0}^{(M)}$ of the reduced convex hull of $S'$ into its connected components, ordered in such a way that the map $\phi_{i, j, n}$ decomposes into maps $\phi_{i, j, n}^{(l)} : \Sigma_{S, 0}^{(l)} \to \Sigma_{S', 0}^{(l)}$ for $1 \leq l \leq M$ which can be described as follows:
\begin{itemize}

\item the map $\phi_{i, j, n}^{(m)}$ is simply the folding $\phi_{i, j, n}$ restricted to $\Sigma_{S, 0}^{(m)}$;

\item the map $\phi_{i, j, n}^{(m')} : \Sigma_{S, 0}^{(m')} \to \Sigma_{S', 0}^{(m')} = s_j^n(\Sigma_{S, 0}^{(m')})$ is the action of the automorphism $s_j^n \in \PGL_2(K)$ if $S_{m'} \subset I_{i, j} \cup \{j\}$; and 

\item the map $\phi_{i, j, n}^{(m')} : \Sigma_{S, 0}^{(m')} \to \Sigma_{S', 0}^{(m')} = \Sigma_{S, 0}^{(m')}$ is the identity if $S_{m'} \cap I_{i, j} \subseteq \{j\}$.

\end{itemize}

It follows that if $S'$ is a good folding of $S$, then the folding operation acts as the ``same'' good folding of $S_m$ while leaving each $S_{m'}$ for $m' \neq m$ unchanged up to an automorphism in $\PGL_2(K)$.  Thus, any sequence of good foldings $S =: S_0 \to S_1 \to \dots \to S_N$ (as in \Cref{lemma algorithm terminates}) is equivalent (up to automorphisms in $\PGL_2(K)$) to a sequence $S_m =: (S_m)_0 \to (S_m)_1 \to \dots \to (S_m)_N$ of good foldings or images under automorphisms in $\PGL_2(K)$ of each $S_m$.  Since $S$ is superelliptic if and only if it is folding equivalent to an optimal set $\Smin$ by \Cref{thm criterion for being superelliptic}, it now follows that $S$ is superelliptic if and only if each $S_m$ is superelliptic.
\end{proof}

We use these results to investigate the possible outcomes of the problem when $g = 1$ and $g = 2$.

\begin{prop} \label{prop g = 1 converse}

Suppose that $S \subset \proj_K^1$ is a subset of cardinality $4$.  Then the properties of being optimal, superelliptic, and clustered in $\frac{v(p)}{p - 1}$-separated pairs are all equivalent for $S$.

\end{prop}

\begin{proof}

Optimal implies superelliptic by \Cref{lemma optimal is superelliptic}, and superelliptic implies clustered in $\frac{v(p)}{p - 1}$-separated pairs by \Cref{prop clustered in pairs}.  Now assume that the set $S$ is clustered in $\frac{v(p)}{p - 1}$-separated pairs.  The graph $\Sigma_{S, 0}$ has exactly $2$ distinguished vertices by \Cref{prop reduced convex hull}(d) and is connected by \Cref{prop reduced convex hull}(e).  These $2$ distinguished vertices are thus connected by a single edge, and therefore each has valency $1$.  Now by \Cref{cor valency}, the set $S$ is optimal.
\end{proof}

\begin{rmk} \label{rmk g = 2}

When $S \subset \proj_K^1$ is a subset of cardinality $6$ which is clustered in $\frac{v(p)}{p - 1}$-separated pairs, there are three possible graph isomorphism types of $\Sigma_{S, 0}$, as noted in \cite[\S IX.2.5.3]{gerritzen2006schottky}.

In order to satisfy the necesssary condition of being clustered in pairs, there must be at least $2$ even-cardinality clusters of $S$, and it is clear for purely combinatorial reasons that there cannot be more than $3$ even-cardinality clusters of $S$.  In the case that there are $3$ even-cardinality clusters of $S$, the reduced convex hull $\Sigma_{S, 0}$ has exactly $3$ distinguished vertices, one corresponding to each cluster, each of valency $1$; by \Cref{cor valency}, the set $S$ is guaranteed to be superelliptic (when $p = 2$, this is in fact the genus-$2$ case of Kadziela's criterion; see \Cref{rmk kadziela} below).

In the case that there are only $2$ even-cardinality clusters and at least one cardinality-$3$ cluster, by \Cref{prop reduced convex hull}(d)(e), the reduced convex hull $\Sigma_{S, 0}$ has $4$ distinguished vertices and $2$ connected components, each of which has $2$ of the distinguished vertices with a single edge connecting them.  Since each distinguished vertex has valency $1$, again by \Cref{cor valency} (or by combining Propositions \ref{prop components} and \ref{prop g = 1 converse}) the set $S$ is guaranteed to be superelliptic.  (This is the third of the cases described in \cite[\S IX.2.5.3]{gerritzen2006schottky}, in which it is written that ``one finds examples where the group $\Gamma$ does not have the right properties''.  No example is given, and it is unclear why Gerritzen and van der Put got this result when our reasoning shows that $S$ is always superelliptic, \textit{i.e.} always ``has the right properties''.)

In the case that there are only $2$ even-cardinality clusters and no cardinality-$3$ cluster of $S$, the reduced convex hull $\Sigma_{S, 0}$ has $3$ distinguished vertices, but in this case one of the vertices has valency $2$.  This is the only case in which it is possible for the set $S$ to not be superelliptic.  One may check whether $S$ is optimal using \Cref{prop good or bad folding computation} and replace $S$ with the set obtained by the appropriate folding if it is not.  \Cref{ex clustered in pairs not sufficient} falls under this case: see \Cref{ex finding optimal1} below for more details.

\end{rmk}

Using \Cref{cor valency}, we are moreover able to recover a result in the dissertation of Kadziela and a closely related (but much less explicitly framed) result of Gerritzen and van der Put and generalize it to any prime $p$ and over any residue characteristic of $K$.

\begin{rmk} \label{rmk kadziela}

For the case that $p = 2$ and the residuce characteristic of $K$ is not $2$, Gerritzen and van der Put state in \cite[\S IX.2.5.2]{gerritzen2006schottky} (and later prove in \cite[\S IX.2.10]{gerritzen2006schottky}) that a particular geometric criterion is sufficient to guarantee that a set $S$ is $2$-superelliptic (\textit{i.e.} that ``the points are in good position''); one can readily show that this criterion is equivalent to the condition that each distinguished vertex of the reduced convex hull has valency $1$ and is thus given by \Cref{cor valency} for any prime $p$ and any residue characteristic.

Meanwhile, Kadziela in \cite[\S5.2]{kadziela2007rigid} describes a criterion which is much more explicitly stated in terms of the elements of $S$ but which is more or less equivalent to that of Gerritzen and van der Put.  Given a set $S \subset \proj_K^1$ which (after possibly applying an appropriate automorphism in $\PGL_2(K)$) satisfies $a_0 := 0, a_g := 1, b_g := \infty \in S$, Kadziela's criterion is that 
\begin{enumerate}[(i)]
\item we have $v(b_0) > v(a_1) \geq v(b_1) \geq \dots \geq v(b_{g - 1}) > 0$; and 
\item in the notation of Kadziela, we have $d_{i, j} < 1$ for distinct indices $i, j \in \{0, \dots, g - 1\}$, which means that the discs $D_i$ are mutually disjoint for $0 \leq i \leq g - 1$; this in turn is equivalent to $S$ being clustered in pairs and its even-cardinality clusters being given by the subsets $\{a_0 = 0, b_0\}, \dots \{a_{g - 1}, b_{g - 1}\}, \{a_0 = 0, b_0, \dots, a_{g - 1}, b_{g - 1}\} \subset S$.
\end{enumerate}
The main result in this section of Kadziela's disseration is \cite[Theorem 5.7]{kadziela2007rigid}, which states (again when $p = 2$ and the residue characteristic of $K$ is not $2$) that Kadziela's criterion is sufficient for a set $S$ to be $2$-superelliptic; Kadziela's criterion is then used as a hypothesis for a number of other useful computations later in his dissertation, including one which proves the conjecture on \cite[p. 282]{gerritzen2006schottky} for this case.  It is easy to see that Kadziela's criterion is again equivalent to the condition that each distinguished vertex of the reduced convex hull has valency $1$; here the distinguished vertices are simply $v_0, \dots, v_g \in \Sigma_{S, 0}$.

\end{rmk}

\section{An algorithm to find an optimal set} \label{sec 4}

The results of the last section point towards a straightforward algorithm that takes an even-cardinality subset $S \subset \proj_K^1$ as its input, determines whether or not $S$ is superelliptic, and in the case that $S$ is determined to be superelliptic, outputs an optimal set $\Smin$ which is folding equivalent to $S$.  Such an algorithm is described below as \Cref{algo folding}.  The main mechanical processes involved in the steps of this algorithm involve determining whether a multiset is clustered in $\frac{v(p)}{p - 1}$-separated pairs (for which one uses the cluster-theoretic definition given in \Cref{rmk clustered in pairs alternate definition}) and determining inclusions between discs in $\cc_K$ and differences in radii of the discs.  We note that the computations just mentioned can be performed as subtasks relying directly only on knowledge of the distances between elements of $K$.

\subsection{The algorithm and how it works} \label{sec 4 algorithm}

Below, we retain the notation introduced in \Cref{dfn D_i} and adopt the following additional notation.

\begin{dfn} \label{dfn tilde D_j^(i)}

Let $S \subset \proj_K^1$ be a subset which is clustered in $\frac{v(p)}{p - 1}$-separated pairs with all of its associated objects, and assume that $b_g = \infty$.  For any distinct indices $i, j$ such that the minimal odd-cardinality cluster containing $a_i, b_i$ coincides with the minimal odd-cardinality cluster containing $a_j, b_j$, set $D_j^{(i)} = D_j$.  For any distinct indices $i, j$ such that there is an odd-cardinality cluster containing $\{a_i, b_i\}$ and exactly $1$ point in $\{a_j, b_j\}$, write $D_j^{(i)} \subset \cc_K$ for the minimal disc containing $\{a_i, b_i\}$ and exactly $1$ point in $\{a_j, b_j\}$.

Let $D^{(i, j)} \subset \cc_K$ be the minimal disc containing both $D_i$ and $D_j^{(i)}$.

If we have $d(D_j^{(i)}) - d(D^{(i, j)}) > \frac{v(p)}{p - 1}$, write $\tilde{D}_j^{(i)}$ for the (unique) disc containing $D_j^{(i)}$ whose (logarithmic) radius is exactly $\frac{v(p)}{p - 1}$ less than $d(D_j^{(i)})$.  If we have $d(D^{(i, j)}) - d(D_j^{(i)}) \leq \frac{v(p)}{p - 1}$, write $\tilde{D}_j^{(i)}$ for the (unique) disc containing $D_i$ whose (logarithmic) radius is exactly $2d(D^{(i, j)}) - d(D_j^{(i)}) - \frac{v(p)}{p - 1}$.  (Note that under this definition, if the residue characteristic of $K$ is not $p$, then we simply have $\tilde{D}_j^{(i)} = D_j^{(i)}$.)

\end{dfn}

The point of the above definition is that, for the pairs of indices $i, j$ the definition applies to, it computes the closest point in $\hat{\Lambda}_{(j)}$ to $v_i$: see \Cref{lemma tilde D_j^(i)} below.

We are now ready to present our algorithm.

\begin{algo} \label{algo folding}

Given the input of an even-cardinality subset $S \subset \proj_K^1$ with $\infty \in S$, perform the following operations (during which we modify $S$ and allow it to become a multiset).

\begin{enumerate}

\item \label{step 1} If precisely a positive even number of elements of $S$ are repeated, then replace $S$ with its underlying set; stop the algorithm; and return $S$ and ``The input set is redundant.''

\item \label{step 2} Check whether $S$ is clustered in $\frac{v(p)}{p - 1}$-separated pairs.  If it is not, stop the algorithm and return ``The input set is not superelliptic.''  Otherwise, label the elements in the pairs as $\{a_i, b_i\} \subset S$ for $0 \leq i \leq g := \frac{1}{2}\#S - 1$, in such a way that $b_g = \infty$.  Set $i = 0$.

\item \label{step 3} Let $j$ be the (necessarily existing and unique) index such that $\tilde{D}_j^{(i)}$ is defined, that we have $\tilde{D}_j^{(i)} \supsetneq D_i$, and that any other $j'$ with $\tilde{D}_{j'}^{(i)} \supsetneq D_i$ satisfies $\tilde{D}_{j'}^{(i)} \supsetneq \tilde{D}_j^{(i)}$.  If we have 
\begin{equation} \label{eq testing for good or bad folding}
v\Big(\frac{a_l - a_j}{a_l - b_j} - \zeta_p^n \frac{a_i - a_j}{a_i - b_j}\Big) > v\Big(\frac{a_l - a_j}{a_l - b_j}\Big) + \frac{v(p)}{p - 1}
\end{equation}
 for some index $l \neq i, g$ and some $n \in \{1, \dots, p - 1\}$ (where the denominators of the fractions in the above inequality are replaced by $1$ if $b_j = \infty$), then proceed to Step \ref{step 4}.  Otherwise, go to Step \ref{step 5}.

\item \label{step 4} Let $I_{i, j} \subset \{0, \dots, g\}$ be the subset of indices $l$ such that $D_l \subsetneq \tilde{D}_j^{(i)}$.  For each $l \in I_{i, j}$, if $j \neq g$, set 
\begin{equation}
a_l' = \frac{(\zeta_p^n a_j - b_j)a_l + (1 - \zeta_p^n)a_j b_j}{(\zeta_p^n - 1)a_l - (a_j - \zeta_p^n b_j)}, \ \ \ b_l' = \frac{(\zeta_p^n a_j - b_j)b_l + (1 - \zeta_p^n)a_j - b_j}{(\zeta_p^n - 1)b_l - (a_j - \zeta_p^n b_j)},
\end{equation}
 and if $j = g$, set $a_l' = (1 - \zeta_p^n)a_g + \zeta_p^n a_l$ and $b_l' = (1 - \zeta_p^n)a_g + \zeta_p^n b_l$.  Replace $S$ with 
\begin{equation*}
(S \smallsetminus \{a_l, b_l\}_{l \in I_{i, j}}) \cup \{a_l', b_l'\}_{l \in I_{i, j}}
\end{equation*}
 and return to Step \ref{step 1}.

\item \label{step 5} If $i \neq g - 1$, replace $i$ with $i + 1$ and return to Step \ref{step 3}.  Otherwise, set $\Smin = S$; stop the algorithm; and return $\Smin$ and ``The input set is superelliptic.''

\end{enumerate}

\end{algo}

A rough idea of what the above algorithm does to an input set $S$ is that it keeps performing good or bad foldings on $S$ in order to bring $S$ closer to being an optimal set, returning to the beginning after performing each (good or bad) folding and detecting immediately in Steps \ref{step 1} and \ref{step 2} if the folding was bad; if we reach a point where there are no more good or bad foldings, the original input set $S$ will have been replaced by an optimal set $\Smin$ (as is guaranteed by \Cref{lemma algorithm terminates}), and only in this situation do we reach Step \ref{step 5}.  More precisely, the effectiveness of \Cref{algo folding} at discerning whether an input is superelliptic, and at producing an optimal set which is folding equivalent if so, is given by the below theorem.

\begin{thm} \label{thm algorithm}

\Cref{algo folding} terminates after executing a finite number of steps.  The output may be interpreted as follows.

\begin{enumerate}[(a)]

\item If \Cref{algo folding} terminates at Step \ref{step 5} by returning a subset $\Smin \subset \proj_K^1$ and ``The input set is superelliptic.'', then the original input set $S$ is superelliptic, and the set $\Smin$ is optimal and folding equivalent to $S$.

\item If \Cref{algo folding} terminates at Step \ref{step 1} by returning a subset $S' \subset \proj_K^1$ and ``The input set is redundant.'', then the original input set $S$ is not superelliptic as it violates property (ii) (but not necessarily property (i)) in \Cref{dfn superelliptic set}, and the lower-cardinality set $S'$ has the same associated subgroups $\Gamma_0 \lhd \Gamma \lhd \PGL_2(K)$ as $S$ does.

\item If \Cref{algo folding} terminates at Step \ref{step 2} by returning ``The input set is not superelliptic.'', then the original input set $S$ violates property (i) in \Cref{dfn superelliptic set} and is thus not superelliptic.

\end{enumerate}

\end{thm}

In order to demonstrate how \Cref{algo folding} works and prove the above theorem, we need some results which are useful in their own rights as they provide elementary methods of computing properties of reduced convex hulls and foldings.

\begin{lemma} \label{lemma tilde D_j^(i)}

With the above set-up, given distinct indices $i, j$ such that $\tilde{D}_j^{(i)}$ can be defined as in \Cref{dfn tilde D_j^(i)}, the point $\tilde{v}_j^{(i)} := \eta_{\tilde{D}_j^{(i)}} \in \Sigma_{S, 0}$ is the closest point in $\hat{\Lambda}_{(j)}$ to $v_i$.

\end{lemma}

\begin{proof}

This follows from a tedious but straightforward argument directly using definitions, \Cref{rmk paths in berk}, and the observations that $v := \eta_{D_j^{(i)}}$ is the (unique thanks to \Cref{prop reduced convex hull}(d)) point in $\Lambda_{(j)}$ lying in the same component of $\Sigma_{S, 0}$ as the point $v_i = \eta_{D_i}$, and that the path $[v_i, v] \subset \Sigma_{S, 0}$ passes through $\eta_{D^{(i, j)}} = v_i \vee v$.
\end{proof}

\begin{lemma} \label{lemma valuations}

For any $r \in \rr$, write $D(r) \subset \cc_K$ for the disc $\{z \in \cc_K \ | \ v(z) \geq r\}$, and for each pair of distinct points $a, b \in \proj_{\cc_K}^1$, let $\Lambda_{a, b} := [\eta_a, \eta_b] \subset \Berk$ denote the axis connecting their corresponding points of Type I, and define the point $\eta_a \vee \eta_b \in \Berk$ as in \Cref{rmk paths in berk}.

\begin{enumerate}[(a)]

\item Given any point $a \in \proj_{\cc_K}^1 \smallsetminus \{0, \infty\}$, the closest point in the axis $\Lambda_{0, \infty}$ to $\eta_a$ is $\eta_{D(v(a))}$.

\item Given any distinct points $\eta, \eta' \in \Berk$, the point $\eta \vee \eta' \in [\eta, \eta']$ has minimal distance to the axis $\Lambda_{0, \infty}$ among points in the path $[\eta, \eta']$.

\item Given any distinct points $a, b \in \proj_{\cc_K}^1 \smallsetminus \{0, \infty\}$, writing $r = \min\{v(a), v(b)\}$, we have 
\begin{equation} \label{eq valuations}
v(a - b) = r + \delta(\eta_a \vee \eta_b, D(r)) = r + \delta(\Lambda_{a, b}, \Lambda_{0, \infty}).
\end{equation}

\end{enumerate}

\end{lemma}

\begin{proof}

Parts (a) and (b) are both elementary observations that follow from \Cref{rmk paths in berk} (see also \cite[Example 6.30]{benedetto2019dynamics}).  Meanwhile, in the situation of part (c), if we have $v(a) \neq v(b)$, then we clearly have $\eta_a \vee \eta_b = \eta_{D(r)} \in \Lambda_{0, \infty}$, and the equations in (\ref{eq valuations}) follow immediately, so let us assume that we instead have $v(a) = v(b) = r$.  The closest point in $\Lambda_{0, \infty}$ to both $\eta_a$ and $\eta_b$ is $\eta_{D(r)}$, so (by uniqueness of paths as in \Cref{rmk paths in berk}) the closest point in $\Lambda_{0, \infty}$ to $\Lambda_{a, b}$ is $\eta_{D(r)}$.  This proves the second equality in (\ref{eq valuations}).

Now let $D_{a, b} \subset \cc_K$ be the minimal disc containing both $a$ and $b$.  We observe that $\eta_{D_{a, b}} = \eta_a \vee \eta_b$; that its (logarithmic) radius is $d(D_{a, b}) = v(a - b)$; and that we have the containment $D_{a, b} \subseteq D(r)$.  Now by our definition of the metric on $\Berk$, we have $\delta(\eta_{a} \vee \eta_{b}, \eta_{D(r)}) = d(D_{a, b}) - d(D(r)) = v(a - b) - r$, and part (c) is proved.
\end{proof}

\begin{prop} \label{prop good or bad folding computation}

Let $S \subset \proj_K^1$ be a subset which is clustered in $\frac{v(p)}{p - 1}$-separated pairs with all of its associated objects.  Let $j$ be an index and $v \in \hat{\Lambda}_{(j)} \cap \Sigma_{S, 0}$.

There exists a good or bad folding $\phi_{i, j, n}$ of $S$ if and only if, for some index $l \notin I_{i, j}$ (with $I_{i, j} \subset \{0, \dots, g\}$ defined as in \Cref{dfn folding}) and some (any) $c_i \in \{a_i, b_i\} \smallsetminus \{\infty\}$ and $c_l \in \{a_l, b_l\} \smallsetminus \{\infty\}$, we have 
\begin{equation} \label{eq computing good or bad folding}
v\Big(\frac{c_l - a_j}{c_l - b_j} - \zeta_p^n \frac{c_i - a_j}{c_i - b_j}\Big) > v\Big(\frac{c_l - a_j}{c_l - b_j}\Big) + \frac{v(p)}{p - 1},
\end{equation}
 where the denominators of the fractions in the above inequality are replaced by $1$ if $b_j = \infty$.

\end{prop}

\begin{proof}

It follows from \Cref{prop pgl2 action on berk} that we may apply any automorphism $\sigma \in \PGL_2(K)$ to $S$ without affecting the structure of $\Sigma_S$, and so we apply the automorphism $\sigma : z \mapsto \frac{z - a_j}{z - b_j}$ (resp. $\sigma : z \mapsto z - a_j$) if $b_j \neq \infty$ (resp. $b_j = \infty$), so that the points $a_j, b_j \in \proj_K^1$ are sent to $0, \infty$ respectively, noting that $\sigma(v_i)$ corresponds to the minimal disc containing the new $\sigma(a_i), \sigma(b_i) \in K$.  After replacing $S$ (for the rest of the proof) with $\sigma(S)$ (and replacing each $a_l, b_l$ respectively with $\sigma(a_l), \sigma(b_l)$), we choose $s_j \in \PGL_2(K)$ to be the order-$p$ element which fixes $a_j = 0$ and $b_j = \infty$ given by $z \mapsto \zeta_p z$, and the inequality in (\ref{eq computing good or bad folding}) becomes simply 
\begin{equation} \label{eq computing good or bad folding simpler}
v(c_l - s_j^n(c_i)) > v(c_l) + \frac{v(p)}{p - 1}.
\end{equation}
This implies that we have $v(c_l) = v(s_j^n(c_i)) = v(\zeta_p^n c_i) = v(c_i)$, which then implies that we have $\delta(\eta_{c_l} \vee \eta_{s_j^n(c_i)}, \Lambda_{(j)}) > \frac{v(p)}{p - 1}$ by \Cref{lemma valuations}(c) and using the notation given in the statement of that lemma.  Then by definition of the subspace $\hat{\Lambda}_{(j)} \subset \Berk$, we have $\eta_{c_l} \vee \eta_{s_j^n(c_i)} \notin \hat{\Lambda}_{(j)}$.  Let $v \in \Lambda_{(j)} \cap \Sigma_{S, 0}$ be the distinguished vertex corresponding to the disc $D(v(c_i)) = D(v(c_l)) \subset \cc_K$, and let $\tilde{v}$ be the closest point in $\hat{\Lambda}_{(j)}$ to $v_i$; we have $v_i, v_l \notin \hat{\Lambda}_{(j)}$.  By letting $c_i$ (resp. $c_l$) range over $\{a_i, b_i\}$ (resp. $\{a_l, b_l\}$), we see that $v_l \vee s_j^n(v_i) \notin \hat{\Lambda}_{(j)}$.  Using \Cref{lemma valuations}(b), it is easy to deduce from this that the point $v_l \vee s_j^n(v_i)$ lies in both paths $[\tilde{v}, v_l]$ and $[\tilde{v}, s_j^n(v_i)]$.  Letting $\tilde{v}_\epsilon = v_l \vee s_j^n(v_i)$, property (ii) in \Cref{dfn folding} follows for the folding $\phi_{i, j, n}$ of $S$, and by definition it is a good or bad folding.

Assume conversely that a good or bad folding $\phi_{i, j, n}$ of $S$ exists.  Then \Cref{cor bad foldings satisfy (ii)} says that property (ii) in \Cref{dfn folding} holds.  We therefore have $[\tilde{v}_{-\epsilon}, \tilde{v}_\epsilon] \subset [v_i, s_j^n(v_i)] \cap \Sigma_{S, 0}$, which implies that there is an index $l$ such that we have $[\tilde{v}, \tilde{v}_\epsilon] \subset [\tilde{v}, w]$, where $w$ is the (unique thanks to \Cref{prop reduced convex hull}(d)) distinguished vertex lying both in the same component of $\Sigma_{S, 0}$ and in $\hat{\Lambda}_{(l)}$.  It is clear from the definition of $I_{i, j}$ that it does not contain $l$ because otherwise there would be backtracking in the journey from $v_i$ to $\tilde{v}_\epsilon \in [\tilde{v}, w]$.  Property (ii) also implies that the same points $v \in \Lambda_{(j)}$ and $\tilde{v} \in \hat{\Lambda}_{(j)}$ are the closest points in $\Lambda_{(j)}$ and $\hat{\Lambda}_{(j)}$ respectively to all of the points $\eta_{s_j^n(a_i)}, \eta_{s_j^n(b_i)}, \eta_{a_l}, \eta_{b_l} \in \Berk$, so that by \Cref{lemma valuations}(b), we have $\eta_{c_l} \vee \eta_{s_j^n(c_i)} \notin \hat{\Lambda}_{(j)}$ for each $c_l \in \{a_l, b_l\} \smallsetminus \{\infty\}$ and $c_i \in \{a_i, b_i\} \smallsetminus \{\infty\}$.  Now, using the definition of $\hat{\Lambda}_{(j)}$ and \Cref{lemma valuations}(c) as above but in reverse, we get the desired inequality in (\ref{eq computing good or bad folding simpler}).
\end{proof}

\begin{proof}[Proof (of \Cref{thm algorithm})]

\Cref{algo folding} may be interpreted as taking an input set $S_0$ and replacing it (in renditions of Step \ref{step 4}) with a sequence of multisets $S_1, S_2, \dots \subset \proj_K^1$ obtained as images of foldings; on making each replacement, we return to Step \ref{step 1}.

We first show that each $S_{m + 1}$ is the image of a good or bad folding of $S_m$ for each $m \geq 0$.  Via \Cref{lemma tilde D_j^(i)}, we see that the algorithm proceeds to run through the indices $0 \leq i \leq g - 1$ and (in Step \ref{step 3}) assigns to each one the index $j$ such that $\tilde{v}_i^{(j)} > v_i$ (where $\tilde{v}_i^{(j)} \in \Sigma_{S, 0}$ is the closest point in $\hat{\Lambda}_{(j)}$ to $v_i$) and there is no other point of the form $\tilde{v}_{j'}^{(i)}$ (for $j' \neq j, i$) in the interior of the path $[v_i, \tilde{v}_j^{(i)}]$.  One sees that this index $j$ exists from the fact that, as we have assumed that $b_g = \infty$, we may take $j = g$ and observe that we have $D^{(i, j)} = D_j^{(i)} \supsetneq D_i$.  One sees moreover that it is unique from the basic properties of the partial order observed above.  Now the algorithm proceeds, in Steps \ref{step 3}, to test whether or not for each $n \in \zz \smallsetminus p\zz$ the folding $\phi_{i, j, n} : S_m \to S_{m + 1}$ is a good or bad folding through a direct use of \Cref{prop good or bad folding computation}; if such an exponent $n$ is found for the pair of indices $(i, j)$, we proceed to Step \ref{step 4}.

In Step \ref{step 4}, our definition of $I_{i, j}$ can be verified to be the one given in \Cref{dfn folding}, and one checks using a straightforward computation (following the arguments in the proof of \Cref{prop order-p elliptic}(a)) that the formulas given for the points $a_l', b_l' \in \proj_K^1$ make $a_l' = s_j^n(a_l)$ and $b_l' = s_j^n(b_l)$, so that $S_{m + 1} := (S_m \smallsetminus \{a_l, b_l\}_{l \in I_{i, j}}) \cup \{a_l', b_l'\}_{l \in I_{i, j}}$ is the image of the folding $\phi_{i, j, n}$ of $S_m$.  The algorithm then begins again from Steps \ref{step 1}--\ref{step 2} with $S_{m + 1}$.  We note that the values of $a_g, b_g = \infty$ do not change in this process, and so in particular, we still have $\infty \in S_{m + 1}$ and the instruction in Step \ref{step 2} to label it as $b_g$ makes sense.

If, on the other hand, the computations in Step \ref{step 3} show that there is no good or bad folding of $S$ with respect to $(i, j)$, we move to Step \ref{step 5}.  If $i \neq g - 1$ (in other words, if we have not exhausted all of the distinguished vertices $v_i$ apart from $v_g$), we increase $i$ by $1$.  If instead we have reached this point with $i = g - 1$, then it follows from \Cref{rmk folding change of i} that there is no good or bad folding of $S_m$, and so the set $S_m$ is optimal by definition.  \Cref{lemma optimal is superelliptic} now says that the set $S_m$ is superelliptic, which implies by \Cref{prop folding equivalence properties} that the original input set $S_0$ is superelliptic.  This is reflected by the termination of the algorithm at Step \ref{step 5} with an output of $\Smin := S_m$.  This proves part (a) of the theorem.

Meanwhile, Step \ref{step 1} tests whether $S_m$ is a set (without repeated elements) and if not, finds the parity of the number of repeated elements.  If the multiset $S_m$ is not a set, then we must have $n \geq 1$ and that $S_m$ by definition is the image of a bad folding of $S_{m - 1}$; thus by \Cref{lemma bad foldings are bad} combined with \Cref{prop folding equivalence properties}, the input set $S_0$ is not superelliptic, as is reflected by the termination and output of the algorithm at this point.  If $S_m$ has an even number of repeated elements, then replacing $S_m$ with its underlying set yields a set which, if clustered in $\frac{v(p)}{p - 1}$-separated pairs, has the same associated subgroups $\Gamma \lhd \Gamma_0 < \PGL_2(K)$ as the original $S_m$ had.  Step \ref{step 2}, meanwhile, tests whether $S_m$ is separated in $\frac{v(p)}{p - 1}$-separated pairs; if it is not, then again the input set $S_0$ is not superelliptic by \Cref{prop clustered in pairs} combined with \Cref{prop folding equivalence properties}.  We have thus proved parts (b) and (c) of the theorem.

Finally, \Cref{lemma algorithm terminates} implies that \Cref{algo folding} terminates after a finite number of steps.
\end{proof}

\subsection{Examples solved using the algorithm} \label{sec 4 examples}

We finish by using \Cref{algo folding} to determine whether a set $S$ is superelliptic and, if so, to compute an optimal set $\Smin$ which is folding equivalent to $S$, in two examples.

\begin{ex} \label{ex finding optimal1}

Let us show through \Cref{algo folding} that the subset 
\begin{equation*}
S = \{a_0 = 7, b_0 = 12, a_1 = 0, b_1 = 5, a_2 = 1, b_2 = \infty\} \in \proj_{\qq_5}^1
\end{equation*}
 from \Cref{ex clustered in pairs not sufficient} is not $2$-superelliptic.  We have (with $v : \qq_5^\times \to \zz$ the usual $5$-adic valuation) 
\begin{equation*}
D_0 = \{z \in \cc_{\qq_5} \ | \ v(z - 2) \geq 1\}, \ D_1 = \{z \in \cc_{\qq_5} \ | \ v(z) \geq 1\}, \ D_2 = \{z \in \cc_{\qq_5} \ | \ v(z) \geq 0\},
\end{equation*}
 and that the points of Type II corresponding to these discs are the only distinguished points of $\Sigma_{S, 0}$ by \Cref{prop reduced convex hull}(d).  The only one of these discs containing $D_0$ is $D_2$, and so if we set $i = 0$ as in Step \ref{step 2} of the algorithm, we get $j = 2$ in Step \ref{step 3}.  Now putting $l = 1$ and plugging the appropriate elements into (\ref{eq testing for good or bad folding}), we verify the inequality in (\ref{eq testing for good or bad folding}) by getting $v(5) > v(6)$, and so we proceed to Step \ref{step 4}.  It is clear from the definition of $I_{i = 0, j = 2}$ in this step that we have $I_{0, 2} = \{0\}$.  Following the formulas in Step \ref{step 4}, we get $a_0' = -5$ and $b_0' = -10$, and so we must replace $S$ with 
\begin{equation*}
(S \smallsetminus \{a_0, b_0\}) \cup \{a_0', b_0'\} = \{-5, -10, 0, 5, 1, \infty\}
\end{equation*}
 and return to Step \ref{step 1}.  But this new set is clearly \textit{not} clustered in pairs, as the subset $\{-5, -10, 0, 5\}$ is a cluster of cardinality $4$ that does not contain any cluster of cardinality $2$.  Therefore, on returning to Step \ref{step 2} of \Cref{algo folding}, the algorithm ends by telling us that $S$ is not $2$-superelliptic.

We further mention that the method by which the non-loxodromic element was produced in \Cref{ex clustered in pairs not sufficient} could be understood as follows.  In terms of our original elements $s_0, s_1, s_2 \in \PGL_2(K)$, the order-$2$ elements fixing pairs of points in the new set $S$ are given by $s_2 s_0 s_2^{-1} = s_2 s_0 s_2, s_1, s_2 \in \PGL_2(K)$; the respective pairs of points fixed by $s_2 s_0 s_2$ and $s_1$ are $\{a_0' = -5, b_0' = -10\}$ and $\{a_1 = 0, b_1 = 5\}$.  Now since the disc $E := D_{a_1, b_1}$ lies in both the axis connecting $\eta_0$ to $\eta_5$ and the one which connects $\eta_{-5}$ to $\eta_{-10}$, following the reasoning in the proof of \Cref{prop clustered in pairs}, the product $(s_1)(s_2 s_0 s_2)$ must then fix $\eta_E \in \Berk$ and is therefore not loxodromic.

The (bad) folding that we have performed is shown in \Cref{fig folding example1} below, in which the nodes shown in the diagram are the distinguished vertices of each reduced convex hull; here we fix $v_0, v_1, v_2$ to be the distinguished vertices associated to the original set $S$.

\end{ex}

\begin{figure}[b]
\centering
\includegraphics[height=1.5cm]{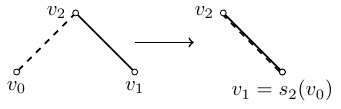}
\caption{The reduced convex hull of our original set $S$ followed by the convex hull of the of the folding performed in \Cref{ex finding optimal1}.}
\label{fig folding example1}
\end{figure}

\begin{ex} \label{ex finding optimal2}

Let us apply \Cref{algo folding} to the set $S = \{\frac{1336}{3}, -355, -110, 86, 0, 7, 1, \infty\}$ over $K := \qq_7$ (with the usual $7$-adic valuation $v : \qq_7^\times \to \zz$) with $p = 2$.  This set has no repeated elements, and it is easy to see that it is clustered in ($\frac{v(2)}{2 - 1} = 0$-separated) pairs, which are 

\noindent $\{a_0 := \tfrac{1336}{3}, b_0 := -355\}, \{a_1 := -110, b_1 := 86\}, \{a_2 := 0, b_2 := 7\}, \{a_3 := 1, b_3 := \infty\}$.  We have 
\begin{equation} \label{eq discs in finding optimal2}
D_0 = -355 + 7^4\mathcal{O}, \ D_1 = -12 + 7^2\mathcal{O}, \ D_2 = 7\mathcal{O}, \ D_3 = \mathcal{O},
\end{equation}
 where $\mathcal{O} \subset \cc_K$ is the ring of integral elements, and that the points of Type II corresponding to these discs are the only distinguished points of $\Sigma_{S, 0}$ by \Cref{prop reduced convex hull}(d).  From this we see that the discs containing $D_0$ are $D_1$ and $D_3$, with $D_1$ being the minimal disc with this property, so we get $j = 1$ in Step \ref{step 3}.  Now putting $l = 1$ and plugging the appropriate elements into (\ref{eq testing for good or bad folding}), we verify the inequality in (\ref{eq testing for good or bad folding}) by getting $v(\frac{224}{935}) > v(\frac{-111}{85})$, and so we proceed to Step \ref{step 4}.  It is clear from the definition of $I_{i = 0, j = 1}$ in this step that we have $I_{0, 1} = \{0\}$.  Following the formulas in Step \ref{step 4}, we get $a_0' = 9$ and $b_0' = -40$, and so we must replace $S$ with 
\begin{equation}
(S \smallsetminus \{a_0, b_0\}) \cup \{a_0', b_0'\} = \{9, -40, -110, 86, 0, 7, 1, \infty\}
\end{equation}
 and return to Step \ref{step 1}.
 
Now it is easy to verify that this new $S$ has no repeated elements and is still clustered in pairs, which we name as before except that the pair $\{\frac{1336}{3}, -355\}$ has been replaced by $\{a_0 := 9, b_0 := -40\}$.  Now the only disc among those listed in (\ref{eq discs in finding optimal2}) containing $D_0$ is $D_3 = \mathcal{O}$, so we get $j = 3$ in Step \ref{step 3}.  Now putting $l = 2$ and plugging the appropriate elements into (\ref{eq testing for good or bad folding}), we verify the inequality in (\ref{eq testing for good or bad folding}) by getting $v(7) > v(-1)$, and so we proceed to Step \ref{step 4}.  We observe from a careful application of the definition of $I_{i = 0, j = 3}$ in this step that we have $I_{0, 3} = \{0, 1\}$.  Following the formulas in Step \ref{step 4}, we get $a_0' = -7$, $b_0' = 42$, $a_1' = 112$, and $b_1' = -84$, and so we must replace $S$ with 
\begin{equation}
(S \smallsetminus \{a_0, b_0, a_1, b_1\}) \cup \{a_0', b_0', a_1', b_1'\} = \{-7, 42, 112, -84, 0, 7, 1, \infty\}
\end{equation}
 and again return to Step \ref{step 1}.
 
\begin{figure}[b]
\centering
\includegraphics[height=4cm]{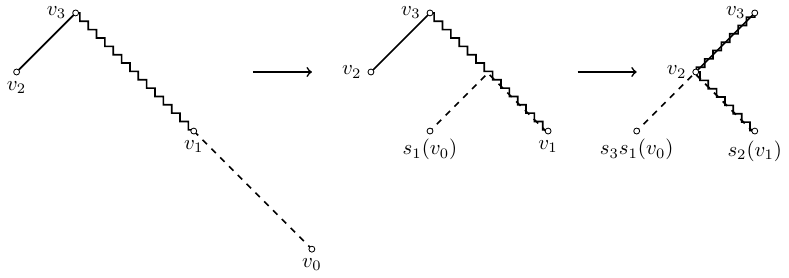}
\caption{The reduced convex hull of our original set $S$ followed by the convex hulls of the sequence of its foldings performed in \Cref{ex finding optimal2}.}
\label{fig folding example2}
\end{figure}
 
Again, this new $S$ visibly has no repeated elements and is still clustered in pairs.  We claim that the set $S$ is optimal.  This can be seen by directly running through the algorithm and seeing that for each $i = 0, 1, 2$, with $j$ chosen appropriately, the inequality in (\ref{eq testing for good or bad folding}) is not verified (specifically the algorithm calls us to check for the ordered pairs $(i, j) = (0, 2), (1, 2), (2, 3)$).  At $i = 2$, we therefore reach Step \ref{step 5}, and the algorithm tells us that $\Smin := S$ is optimal and that our original input set $\{\frac{1336}{3}, -355, -110, 86, 0, 7, 1, \infty\}$ is $2$-superelliptic.

The sequence of two (good) foldings that we performed is shown in \Cref{fig folding example2} below, in which the nodes shown in the diagram are the distinguished vertices of each reduced convex hull; here we fix $v_0, \dots, v_3$ to be the distinguished vertices associated to the original set $S$.

\end{ex}






\bibliographystyle{plain}
\bibliography{bibfile}

\end{document}